\documentclass[11pt]{amsart}
\usepackage{amsmath}
\usepackage{amssymb}
\usepackage{graphicx}
\usepackage{cite}
\usepackage{epsfig}
\usepackage{latexsym}
\usepackage[mathcal]{eucal}
\usepackage{lscape}
\usepackage{comment}
\usepackage{color}
\usepackage{hyperref}
\usepackage{caption}
\usepackage{subcaption} 
\usepackage{todonotes}
\usepackage{atveryend}
\usepackage{float}
\makeatletter
\renewcommand\subsection{\@startsection{subsection}{2}%
\z@{-.5\linespacing\@plus-.7\linespacing}{.5\linespacing}%
{\normalfont\scshape}} 

\renewcommand\appendix{\par
  \setcounter{section}{0}
  \setcounter{subsection}{0}
  \setcounter{figure}{0}
  \setcounter{table}{0}

\renewcommand\thefigure{\Alph{section}\arabic{figure}}

\renewcommand\thetable{\Alph{section}\arabic{table}}
}

\makeatletter
\let\origcitation\citation
\AtEndDocument{\def\mycites{\@gobble}%
  \def\citation#1{\g@addto@macro\mycites{,#1}\origcitation{#1}}}
\AtVeryEndDocument{\typeout{***^^JCited keys: \mycites^^J***}}
\makeatother

\usepackage{comment}
\hypersetup{
    colorlinks=true,
    linkcolor=blue,
    filecolor=magenta,
    urlcolor=cyan,
    citecolor=red,
}
\newcommand{\dom}{{\rm dom\,}}

\newcommand{\ds}{\displaystyle}

\newcommand{\nexto}{\kern -0.54em}
\newcommand{\R}{\mathbb{R}}
\newcommand{\dN}{\mathbb{N}}

\newtheorem{theorem}{Theorem}[section]
\newtheorem{corollary}[theorem]{Corollary}
\newtheorem{lemma}[theorem]{Lemma}
\newtheorem{proposition}[theorem]{Proposition}
\theoremstyle{definition}
\newtheorem{definition}[theorem]{Definition}

\newtheorem{remark}[theorem]{Remark}
\newtheorem{algorithm}{Algorithm}[section]

\newtheorem{fact}{Fact}[section]

\renewcommand{\theequation}{\arabic{equation}}

\newcommand{\argmin}{\operatornamewithlimits{argmin}}

\newcommand{\Argmin}{\operatornamewithlimits{Argmin}}

\numberwithin{equation}{section}

\title[A primal--dual algorithm for optimal control]{A primal--dual algorithm as applied to optimal control problems}

\author[R. S. Burachik]{Regina S. Burachik}

\address[R. S. Burachik]{Mathematics, UniSA STEM, University of South Australia, Mawson Lakes, S.A. 5095, Australia}
\email{{\tt regina.burachik@unisa.edu.au}}

\author[C. Y. Kaya]{C. Yal{\c c}{\i}n Kaya}
\address[C. Y. Kaya]{Mathematics, UniSA STEM, University of South Australia, Mawson Lakes, S.A. 5095, Australia}
\email{\tt yalcin.kaya@unisa.edu.au}

\author[X. Liu]{Xuemei Liu}
\address[X. Liu]{Mathematics, UniSA STEM, University of South Australia, Mawson Lakes, S.A. 5095, Australia}
\email{\tt xuemei.liu@mymail.unisa.edu.au.}


\subjclass[2010]{????????}

\begin{document}

\begin{abstract}
We use a primal--dual technique for solving infinite dimensional problems arising from optimal control. Namely, we solve as examples a control-constrained double integrator optimal control problem and the challenging control-constrained free flying robot optimal control problem by means of the primal--dual scheme. The algorithm we use is an epsilon-subgradient method that can also be interpreted as a penalty function method.  We provide extensive comparisons of our approach with a traditional numerical approach.
\end{abstract}

\maketitle


\section{Introduction}
Let $U$ be a reflexive Banach space, $H$ a Hilbert space and let $K\subset U$. Consider the following infinite dimensional equality constrained optimization problem:
\begin{equation}\label{eq:generalPr}
\tag{P}    \min_{u\in U} \varphi(u)\;\; {\rm  s.t. }\;\; u\in K, \;\;\;\; h(u)=0\,,
\end{equation}
where $\varphi:U\to\R\cup\{\infty\}$ is lower semi-continuous  and $h:U\to H$ is a given function. In general, \eqref{eq:generalPr} is not convex, so if we are to use duality to solve it, we need to use an augmented Lagrangian approach. The type of Lagrangian we consider for addressing Problem \eqref{eq:generalPr} is an extension of the one in \cite{BFK2014,BKdsg} to infinite dimensions. The infinite dimensional case was recently studied in \cite{BurachikXuemei2023}, where  Problem \eqref{eq:generalPr} is addressed via a primal-dual scheme where the augmented Lagrangian $\ell:U\times\R_+\to\R$ is defined as follows
\begin{equation}\label{eq:generalLagrn}
    \ell(u,c):=\inf_{v\in K}\varphi(v)-\langle u,A(h(v))\rangle+ c\,\sigma(h(v))\,,
\end{equation}
where $u\in U$, $c\ge 0$, $A: H\to H$, is a suitable map, and $\sigma:H\to\R_+$ verifies $\sigma(v)=0$ if and only if $v=0$. The resulting primal-dual scheme is paired in \cite{BurachikXuemei2023} with an epsilon subgradient technique that solves the dual problem.  An important advantage of \eqref{eq:generalLagrn} is that, unlike other available versions (such as the ones proposed in \cite{Gdsg,BGKIdsg,BKMinexact2010,BIMdsg,BIMinexact2013}), the use of this $\ell$ induces a penalty function method for the choice $A=0$.  The aim of the present paper is to exploit this feature, and to use it for solving optimal control problems. More precisely, we will take $A=0$ and $\sigma$ as a suitable norm in $H$. Since it can be seen as a penalty approach, we denote the resulting epsilon-subgradient method as a primal--dual penalty (PDP) algorithm.

Compared with the classical penalty method, the PDP algorithm uses a subgradient direction as its update rule (for the penalty parameter) and improves the dual values in each iteration.  

The classical advantages of the augmented Lagrangian scheme here are that, (i) even when the original problem is not convex, the dual problem is, and hence it can be solved by standard techniques from convex analysis, (ii) there is no gap between primal and dual optimal values, and (iii) solving the dual problem provides a primal solution.

The fact that the primal--dual approach using the augmented Lagrangian \eqref{eq:generalLagrn} enjoys the advantages listed in (i)--(iii) has been proved in \cite{BurachikXuemei2023}. The interested reader can also see \cite{Dol2018,Dol20182} for an excellent introduction on different types of Lagrangians and their applications in solving various kinds of problems. We note that a different type of augmented Lagrangian technique has been used for solving finite-dimensional problems in \cite{Burachik Kaya Price 2021}.

Our first aim is theoretical, and it consists of determining a wide enough family of problems \eqref{eq:generalPr} such that the PDP algorithm is well defined when applied to \eqref{eq:generalPr}. This is needed so as to ensure that properties (i)--(iii) will hold for our scheme. We establish this fact in Theorem \ref{th:example_verify_H}. 

Our second aim is practical, and it is to show that we can apply the PDP algorithm to solve challenging optimal control problems. We achieve this aim by addressing two optimal control problems that do not have an analytical solution available. These problems are the control-constrained double integrator and the free flying robot~\cite{BY2013,BYir,Sakawa1999,VosMau2006}. While the first one of these problems is convex, the second one is highly nonconvex.

Projection-type methods can be used to solve some optimal control problems (for example the one in~\cite{BBY}), as long as they are convex. Our approach via the PDP algorithm, however, can solve also non-convex instances of these problems. Even when dealing with convex problems, the projection is usually difficult to compute unless we project onto simple sets. When compared with the penalty methods proposed in \cite{DGK1998,DGK2000,DGK2000_2,DT2011}, we note that the latter works consider either simple problems or those with analytical solutions. Moreover, the methods in \cite{DGK1998,DGK2000,DGK2000_2,DT2011} have not been implemented. 

\noindent  The paper is organized as follows. In Section~\ref{sec:Pre}, we give the preliminaries on functional analysis, which help in building our assumptions on Problem \eqref{eq:generalPr}.  Section \ref{ssec:PD} provides our theoretical framework, where we show that the family of problems we address verifies the necessary assumptions (see Theorem \ref{th:example_verify_H}). In this section we recall (i) the properties of the duality framework, as well as  (ii) the definition of the PDP method and its properties, which were established in  \cite{BurachikXuemei2023}. 
In Section~\ref{sec:problemP1}, we define a class of optimal control problems in the format of Problem \eqref{eq:generalPr}, and derive its conditions of optimality. Also in this section we describe our discretization scheme.
In Sections~\ref{sec:DoubleIntegrator} and \ref{sec:FFRobot}, we implement the PDP algorithm for solving the constrained double integrator, and the challenging free-flying robot problem, respectively. In these sections we also show the performance of our approach and compare it with a conventional numerical approach. Section~\ref{sec:conclusion} contains the conclusion and further discussion. To simplify the presentation, longer, or more involved, proofs are given in an \hyperlink{prf:lem_varfi}{Appendix} at the end of our paper.

\section{Preliminaries}\label{sec:Pre}

To determine a general family of problems and show that certain optimal control problems belong to that family, we need to recall a few results from functional analysis, which we do in the next subsection.

\subsection{Some Functional Analysis Tools}
Let $X$ be a reflexive Banach space, $X^*$ its topological dual (i.e., the set all continuous linear functionals from $X$ to $\R$), and $H$ a Hilbert space. We denote by $\langle \cdot, \cdot \rangle$ both the duality product in $X\times X^*$ and the scalar product in $H$. Unless explicitly indicated, we denote by $\|\cdot\|$ the norm of $X$ or $H$.  We use the notation $\R_{+\infty}=\R_{\infty}:=\R\cup\{+\infty\}$. Given a function $g:X\to{\R}_{+\infty}\cup\{-\infty\}$, the {\em effective domain\index{effective domain}} of $g$ is the set $\dom g:=\{x\in X\::\: g(x)< +\infty\}$. We say that $g$ is {\em proper}\index{proper} if $g(x)>-\infty$ and $\dom g\neq \emptyset$.  The set $lev_{g}(\alpha):= \{x\in X\::\: g(x)\le \alpha\}$ is the $\alpha$-{\em level set of}\index{level set} $g$. Given $C\subset X$, the {\em indicator function of $C$}\index{indicator function} is defined as $\delta_C(v):=0$ if $v\in C$ and $+\infty$ otherwise. If $C=\{z\}$ is a singleton, we denote $\delta_{\{z\}}=:\delta_z$. 

The topology induced by the norm (in $X$ or $H$), is called the {\em strong} topology. The weak topology in $X$ (weak topology in $H$) is the coarsest topology that makes all elements  of $X^*$ (all elements of $H^*=H$)  continuous. We will need the following definitions concerning the weak topology. For $A\subset X$, we denote by ${\rm cl\,}A$ the strong closure of $A$ and by ${\overline{A}}^w$ the weak closure of $A$.  In most of what follows, when a topological property is mentioned by its own, this means that the property holds w.r.t. the strong (i.e., the norm) topology. For instance, if we write ``$A$ is closed", we mean ``$A$ is {\sc strongly} closed". If a property holds w.r.t. the weak topology, we will mention the term ``weak" (or ``weakly") explicitly. For instance, we may say weakly closed, (or w-closed), weakly compact (or w-compact), etc.  Recall that a function $\varphi:X\to\R_{+\infty}$ is {\emph {weakly lower semi-continuous}}  (\emph {w-lsc}) when it is lsc w.r.t. the weak topology in $X$. Namely, when $epi\; \varphi$ is w-closed. Let $\{u_n\},\{w_n\}\subset X$, we denote the fact that $\{u_n\}$ converges weakly to $u$ as $u_n\rightharpoonup u$, and the fact that $\{w_n\}$ converges strongly to $w$ as $w_n\to w$.

 We recall next some well-known facts from functional analysis, most of which can be found in \cite{Brezis}. The reader familiar with functional analysis can skip this section, with the exception of Lemma \ref{lem:varfi}, which, to our knowledge, is new.
  
  \begin{fact}\label{lem:Weakly Closed}
 Let $X$ be a reflexive Banach space, $H$ be a Hilbert space. Assume that $K\subset X$ is nonempty. The following hold.
 \begin{itemize}
     \item[(a)] $K\neq \emptyset$ is weakly closed if and only if the indicator function $\delta_K$ is proper and w-lsc.
     \item[(b)] If $K\subset X$ is weakly compact, then it is weakly closed.
     \item[(c)]  A convex subset of $X$ is weakly closed if and only if it is closed.
     \end{itemize}
 \end{fact}
  
We recall the following definitions. 
 
 \begin{definition}[Weak compactness; sequential compactness; coercive] \hfill\break 
 \label{def:weakly_compact_coercive}
Let $X$ be a Banach space, $A\subset X$ and $\varphi : X\to \R_{+\infty}$. 
\begin{itemize}
    \item[(a)]  $A$ is {\em weakly-compact}\index{weakly-compact} when its weak closure, ${\overline{A}}^w$,  is compact w.r.t the weak topology.
    \item[(b)] $A\subset X$ is {\em sequentially-compact}\index{sequentially-compact} (respectively, {\em weakly sequentially-compact}) when every\linebreak $\{x_n\}\subset A$ has a subsequence  converging strongly (respectively, weakly) to a limit in $A$. 
     \item[(c)] The function $\varphi : X\to \R_{+\infty}$ is {\em coercive}\index{coercive} when $\lim_{\|x\|\to \infty} \varphi(x)=+\infty$. 
\end{itemize}
\end{definition}
 
The equivalence between compactness and sequential-compactness in normed spaces allows the use of sequences when dealing with compact sets in $X$. To deal with weakly compact sets in terms of sequences, we will use Eberlein--Smulian theorem \cite[Problem 10(3), p. 448]{Brezis}, recalled next.

\begin{theorem}[Eberlein--Smulian]\label{th:ES}
Let $A\subset X$. Then $A$ is weakly compact if and only if it is weakly sequentially-compact.
\end{theorem}

Next we quote results that connect boundedness, closedness and compactness both in strong and weak topologies. The next result, a corollary of Bourbaki--Alaoglu's theorem, is \cite[Corollary 3.22]{Brezis}. 

\begin{theorem}\label{th:BA1}
Let $E$ be a reflexive Banach space. Let $K\subset  E$ be a bounded,
closed, and convex subset of $E$. Then $K$ is weakly compact.
\end{theorem}

\begin{corollary}\label{cor:CCB}
If $X$ is a Banach space, then every weakly compact set is closed and bounded. Consequently, every weakly convergent sequence must be bounded.
\end{corollary}

We will also need the following two results involving functions defined on $X$. The first one is \cite[Corollary 3.9]{Brezis}, and is a direct consequence of Fact \ref{lem:Weakly Closed}(c). The second result can be found, e.g., in \cite[Corollary 2.2]{BurachikXuemei2023}.

\begin{theorem}\label{th:BA2}
Assume that $\varphi : X\to \R_{+\infty}$ is convex. Then $\varphi$ is w-lsc if and only if it is lsc.
\end{theorem}

\begin{corollary}\label{cor:coer}
Let $X$ be a reflexive Banach space and let $\varphi: X\to \R_{+\infty}$ be w-lsc. Then $\varphi$ is coercive if and only if all its level sets are weakly compact. In this situation, all the level sets are closed and bounded.
\end{corollary}

\begin{definition}\label{def:LP}
Let $1\le p<\infty$, ${\mathcal L}^p([0,t_f];\R^m)$ be the Banach space of Lebesgue measurable functions $z:[0,t_f]\to \R^m$, with finite ${\mathcal L}^p$ norm, denoted $\|\cdot\|_{{\mathcal L}p}$, namely,
\[
{\mathcal L}^p([0,t_f];\R^m)=\left\{ z:[0,t_f]\to \R^m \,\,|\,\,
\|z\|_{{\mathcal L}^p}=\ds\left( \int_0^{t_f}\|z(t)\|_2^p\,dt\right)^{1/p}<\infty
\right\},
\]
where $\|\cdot\|_p$ is the $\ell_p$ norm in $\R^m$. $W^{1,2}([0,t_f];\R^m)$ is the Sobolev space of absolutely continuous functions, namely,
\[
{\mathcal W}^{1,2}([0,t_f];\R^m)=\left\{ z\in {\mathcal L}^2([0,t_f];\R^m) \,\,|\,\,
\dot z = dz/dt \in {\mathcal L}^2([0,t_f];\R^m)\,\right\},
\]
endowed with the norm
\[
\|z\|_{{\mathcal W}^{1,2}}=\left(\|z\|_2^2+\|\dot z\|_2^2\right)^{1/2}.
\]
\end{definition}

We will make use of the following result, which is \cite[Proposition 3.5(iv)]{Brezis}.

\begin{proposition}\label{prop:weak-strong}
Let $X$ be Banach space. Consider two sequences $\{u_n\}\subset X,\,\{w_n\}\subset X^*$ and let $u\in X,\,w\in X^*$ be such that $u_n\rightharpoonup u$ (i.e., $\{u_n\}$ converges weakly to $u$), and $w_n\to w$ (i.e., $\{w_n\}$ converges strongly to $w$). Then 
\[
\lim_{n\to \infty} \langle u_n, w_n \rangle= \langle u, w \rangle.
\]
\end{proposition}

The next result follows from Lebesgue's Dominated Convergence Theorem (see, e.g., \cite[Theorem 4.2]{Brezis}).

\begin{theorem}\label{th:LDCT}
Let $f_n:[0,T]\to \R$ for every $n\in \dN$ verifying that $\lim_{n\to \infty} f_n(t)\in \R$ for every $t\in [0,T]$. Assume that there exists $M>0$ such that $|f_n(t)|\le M$ for all $t\in [0,T]$. Define $f:[0,T]\to \R$ such that $f(t):=\lim_{n\to \infty} f_n(t)$ for all $t\in [0,T]$. Then,
\begin{itemize}
    \item[(a)] $f\in {\mathcal L}^1([0,T];\R)$.
    \item[(b)] For every $t\in [0,T]$ we have $\lim_{n\to \infty} \ds\int_0^t f_n(s) ds = \ds\int_0^t f(s) ds$.
\end{itemize}
\end{theorem}

To study the ODE systems of sections \ref{sec:DoubleIntegrator} and \ref{sec:FFRobot}, we will need the lemma below. This lemma  establishes the sequential weak continuity of a family of real-valued functions defined in  $({\mathcal L}^2([0,T];\R))^m$. 

\begin{lemma}\label{lem:varfi}
Let $m\in \mathbb{N},\, T>0$ and $u:=(u^1,\ldots,u^m)\in ({\mathcal L}^2([0,T];\R))^m$. Assume that
\begin{itemize}
    \item[(a)] $\varphi:\R \to \R$ is continuous and globally bounded (i.e., exists $L_1>0$ s.t. $|\varphi(t)|\le L_1$ for all $t\in \R$),
    \item[(b)] Fix $a,b_1,\ldots,b_m\in \R$, and define $\pi(u):[0,T]\to \R$ as
    \[
    \pi(u)(t):=a+ \sum_{j=1}^m b_j \int_{0}^t \left(\int_{0}^r u^j(s)ds\right) dr.
    \]
     \item[(c)] Fix $r,t\in [0,T]$ s.t. $r\le t$ and 
     $j\in \{1,\ldots,m\}$. Define $\eta^{\varphi}_j(\cdot,r),\rho^{\varphi}_j(\cdot,t):({\mathcal L}^2([0,T];\R))^m\to \R$ as 
     \[
\begin{array}{rcl}
   \eta^{\varphi}_j(u,r)  & :=& \ds\int_{0}^r u^j(s)\, \varphi( \pi(u)(s))ds, \\
     &&\\
    \rho^{\varphi}_j(u,t) & := &\ds\int_{0}^t \left(\int_{0}^r u^j(s) \,\varphi( \pi(u)(s))ds\right)dr= \int_{0}^t  \eta^{\varphi}_j(u,r) dr
\end{array}
     \]
\end{itemize}
 If $\{u_k\}\subset ({\mathcal L}^2([0,T];\R))^m$ and $u_k\rightharpoonup u$, then we have
   \begin{equation}
       \label{claim:eta}
       \ds\lim_{k\to \infty}  \eta^{\varphi}_j(u_k,r) =  \eta^{\varphi}_j(u,r), \, \forall \, r \in [0,T]
   \end{equation}
and
\begin{equation}
       \label{claim:rho}
      \ds  \lim_{k\to \infty}   \rho^{\varphi}_j(u_k,t) =  \rho^{\varphi}_j(u,t), \, \forall \, t \in [0,T]
   \end{equation}
Namely, the  functions $\eta^{\varphi}_j(\cdot,r)$ and $ \rho^{\varphi}_j(\cdot,t)$ are w-sequentially continuous for every $r,t\in [0,T]$.
\end{lemma}
\begin{proof}
 See the proof of  Lemma~\ref{lem:varfi} in \hyperlink{prf:lem_varfi}{Appendix}.
\end{proof}

\section{Primal and Dual Problems}\label{ssec:PD}

\subsection{Theoretical Framework for duality}

Following \cite[Section 2.2]{BR2007}, we embed Problem \eqref{eq:generalPr} into a family of parametrized problems using a function that coincides with  $\varphi$ when the parameter is zero. This tool is given next.

\begin{definition}\label{def:PF}
A {\em dualizing parameterization} for \eqref{eq:generalPr} is a
function $f:U\times H\to \bar{\R}$ that verifies $f(u,0)=\varphi(u)$ for all $u\in U$.  
\end{definition}

The next definition is \cite[Definition 5.1]{BR2007} and will be our basic assumption for the dualizing parametrization.

\begin{definition}\label{level-f}
A function $f:U\times H\rightarrow \bar{\R}$ is said to be {\em weakly level-compact} if for each $\bar{z}\in H$ and $\alpha \in \R$
there exist a weakly open neighbourhood $V \subset H$ of $\bar{z}$, and a weakly compact set $B\subset U$, such that
\[ lev_{z,f}(\alpha):=\{u \in U: f(u,z)\leq \alpha \} \subset B\;\; \mbox{for all}\; z \in V.\]
\end{definition}

We next list the basic assumptions of the primal--dual framework.
\begin{itemize}
    \item[\hypertarget{H0}{(H0)}] The objective function $\varphi:
U\to \R_{\infty}$
is proper and w-lsc. 
    \item[\hypertarget{H1}{(H1)}] The function $\varphi$ has weakly compact level sets.
    \item[\hypertarget{H2}{(H2)}] The dualizing parameterization $f$ is proper (i.e., $\dom f\neq \emptyset$ and $f(u,z)>-\infty,$ \linebreak $\forall \, (u,z)\in U\times H$), w-lsc and weakly level-compact (see Definition \ref{level-f}). 
\end{itemize}

The result below will be used in our application to optimal control problems. 

\begin{theorem}[Problem (P) verifies (H0)--(H2)]
\label{th:example_verify_H}
Let $U$ be a reflexive Banach space and let $H$ be a Hilbert space. Consider a  function $\varphi:U\to \R_{+\infty}$, a  set $K\subset U$, and a function $h:U\to H$.  Consider Problem \eqref{eq:generalPr}, i.e.,
\[
({P})\qquad \min \varphi(u)\, \hbox{ s.t. } u\in K \hbox{ and }h(u)=0.
\]
Assume that $S(P)$, the solution set of Problem $({P})$, is nonempty and that the following hold.
\begin{itemize}
    \item[(a)] The objective function $\varphi$ is proper, coercive and w-lsc.
    \item[(b)] For every $z\in H$, the set $K\cap h^{-1}(z)$ is weakly closed, where
     \(
     h^{-1}(z):=\{u\in U\::\:h(u)=z\}. 
    \)
    \item[(c)] The dualizing parameterization $f:U\times H\to \R_{\infty}$ is defined as follows
\[
f(u,z):= \varphi(u)+ \delta_{K}(u) +\delta_{z}(h(u)).
\]
\end{itemize} 
Then Problem $({P})$ verifies assumptions {\rm\hyperlink{H0}{(H0)}--\hyperlink{H2}{(H2)}}.

\end{theorem}
\begin{proof}
We note that \hyperlink{H0}{(H0)} holds automatically by the choice of $\varphi$ in (a). To check \hyperlink{H1}{(H1)}, we need to show that all level sets of $\varphi$ are weakly compact. By (a) $\varphi$ is coercive and w-lsc, so we can apply Corollary \ref{cor:coer} to conclude that all its level sets are weakly compact. Thus \hyperlink{H1}{(H1)} holds. We proceed to check \hyperlink{H2}{(H2)}. From the definition of indicator function we have that
\[
f(u,z)=  \varphi(u)+ \delta_{ K\cap h^{-1}(z)}(u) \ge  \varphi(u)>-\infty.\]
Moreover, $ f(u,z)$ is not identically $+\infty$ because $\emptyset\not=S(P)\subset K\cap h^{-1}(0)$. Hence $ f$ is proper. We proceed now to show that $ f$ is w-lsc. Indeed, it is enough to show that it is the sum of w-lsc functions. By (a), $ \varphi$ is w-lsc. Assumption (b) and Fact \ref{lem:Weakly Closed}(a) imply that $\delta_{ K\cap h^{-1}(z)}$ is weakly-lsc, too. Altogether, $ f$ is the sum of w-lsc functions and hence w-lsc. We proceed now to show that $f$ is weakly level compact. Fix $z_0\in H$ and $W$ any weakly open set containing $z_0$. With the notation of Definition \ref{level-f}, and assumption (c) to write for any $z\in W$
\begin{eqnarray*}
lev_{z,f}(\alpha) &=& \{u \in U: f(u,z)\leq \alpha \} \\
&=& \{u \in K: \varphi(u)\leq \alpha,\, h(u)=z \}\subset \{u \in K:\varphi(u)\leq \alpha\}\subset lev_{\varphi}(\alpha),
\end{eqnarray*}
where  $lev_{\varphi}(\alpha)$ denotes the level set of $\varphi$.  Since $\varphi$ is weakly level compact,  $lev_{\varphi}(\alpha)=:B$ is\linebreak w-compact. Since $B$ does not depend on $z$, the expression above yields
\[
\ds\bigcup_{z\in W}lev_{z,f}(\alpha)\subset B,
\]
which implies that  $f$ is weakly level compact. Therefore \hyperlink{H2}{(H2)} holds.
\end{proof}

\begin{remark}\rm
Problem~\eqref{eq:generalPr} as in Theorem \ref{th:example_verify_H} has been considered in 
\cite[Example 2.1]{BIMinexact2013}, where $h$ is assumed to have a weakly closed graph. In infinite dimensions, this assumption may be too restrictive or hard to establish. We replace it here by the less restrictive assumption (b) which is enough to ensure the w-lsc of the duality parametrization required for \hyperlink{H2}{(H2)} to hold. In later sections, we will show that (a)-(c) in Theorem \ref{th:example_verify_H} hold for our optimal control examples.
\end{remark}

We define next the augmented Lagrangian function, and the resulting problem dual to \eqref{eq:generalPr}. As mentioned in the introduction, this Lagrangian is a particular case of that analyzed in \cite[Section 3]{BurachikXuemei2023}. Namely, we take $A=0$ and $\sigma$ a suitable norm in \eqref{eq:generalLagrn}.

\begin{definition}[Augmented Lagrangian and associated dual problem]
\label{def:basic}
Let $U,H$ and $h$ be as in Problem \eqref{eq:generalPr}. Define $K_0:=K\cap\{u\in U\::\: h(u)=0\}$, i.e., $K_0$ is the constraint set for Problem \eqref{eq:generalPr}. Take $\|\cdot\|$ a norm in $H$. Assume that the
dualizing parameterization $f$ is given by
\begin{equation}\label{eq:Dual-param}
   f(u,z):= \varphi(u) + \delta_{K}(u) +\delta_z(h(u)), 
\end{equation}
and that it satisfies assumption {\rm \hyperlink{H2}{(H2)}}. We consider the following type of augmented {\em Lagrangian} $\ell:U\times \R_+\to \R_{-\infty}$ for Problem \eqref{eq:generalPr}:
\begin{equation}\label{eq:Lagrn}
\ell(u,c):=\inf_{z\in H}\{f(u,z) + c\|z\|\}= \left\{
\begin{array}{cc}
   \{\varphi(u) + c\|h(u)\|\},  & \hbox{ if } u\in K, \\
     & \\
     +\infty & \hbox{ otherwise. }
\end{array}
\right.
\end{equation}

The resulting  {\em dual function} $q: \R_+\to\R_{-\infty}$ is
\begin{equation}
\label{eq:dual_function}
q(c):=\inf_{u\in U} \ell(u,c)= \inf_{u\in K} \varphi(u) + c\|h(u)\|,
\end{equation} 
where we used  \eqref{eq:Lagrn}-\eqref{eq:Dual-param} in the last equality. The \emph{dual problem} of \eqref{eq:generalPr} is given by 
\begin{equation*}\label{eq:D_probm}
(D) \qquad {\rm maximize}\;\;\! q(c) \;\; {\rm  s.t. }\,\, c\ge 0.
\end{equation*}
Denote by $\ds M_P:=\inf_{u\in K_0}\varphi(u)$ and by $\ds M_D:=\sup_{c\ge 0}q(c)$ the optimal values of the primal and dual
problem, respectively.
The primal and dual solution sets are
denoted by $S(P)$ and $S(D)$, respectively.
\end{definition}

\begin{remark}[Finite primal value for Problem (P)] \label{rem:MP finite}  \rm
Assumption \hyperlink{H0}{(H0)} implies that $\varphi$ is proper, so Definition \ref{def:PF} yields $M_P<+\infty$.
\end{remark}

\begin{remark}[Optimal control Examples] \label{rem: OCP examples}  \rm
In our examples, we will always have that $H$, the co-domain of the function $h$, is a finite dimensional Hilbert space. Namely, we will always have that $H:=\R^m$ for some $m\in \dN$. This will allow us to take the (finite dimensional) $\ell_1$ norm as the norm in \eqref{eq:Lagrn}-\eqref{eq:Dual-param}.
\end{remark}

\subsection{Properties of the Primal--Dual Setting}\label{sec:PD properties}

We next present some basic properties of the dual function given in Definition \ref{def:basic}. The proof of the proposition below is standard, and can be found in \cite{BurachikXuemei2023}.

\begin{proposition}[Properties of the dual function]\label{prop:dualProperty} Let $q:\R_+\to \R_{-\infty}$ be the dual function defined in \eqref{eq:dual_function}. The following facts hold.
\begin{itemize}
    \item[(i)] The dual function  $q$ is concave, increasing and
weakly upper-semicontinuous (i.e., $-q$ is w-lsc).
    \item[(ii)]  If
$c_1\in S(D)$, then $c\in S(D)$ for all $c\geq c_1$.
\end{itemize}
\end{proposition}

We state next, adapted to our type of Lagrangian, several properties of the primal dual scheme. We start with strong duality, proved in \cite[Theorem 3.1]{BurachikXuemei2023}. 

\begin{theorem}[Strong duality for $(P)$--$(D)$ framework]\label{th:StrongDual}
Consider the primal-dual problems \eqref{eq:generalPr} and (D). Assume that {\rm\hyperlink{H0}{(H0)}--\hyperlink{H2}{(H2)}} hold.  Suppose that there exists some
$\bar c \in \R_+$ such that $q(\bar c)>-\infty$. Then the zero-duality-gap property holds, i.e. $M_P=M_D$.
\end{theorem}

\begin{definition}[Superdifferential of a concave function]
\label{rSupDiff}
Let $g:\R\to \R_{-\infty}$ be a concave function. The superdifferential
of $g$ at $c_0\in {\rm dom}(g):=\{c\in \R\::\: g(c)>-\infty\}$ is the set $\partial g(c_0)$ defined by
\[\partial g(c_0):=\{v\in \R : g(c)\leq g(c_0) + \langle v,c-c_0 \rangle , \;\;\forall c \in \R\}. \]
\end{definition}

\begin{definition}[Approximations for the primal--dual and Lagrangian]
\label{def:aprrox_Lagn}
Define the set
\begin{equation*}
\label{eq:aproxvalue}
\begin{array}{rcl}
   X(c)  &:= & \{u \in U:
\varphi(u)+ c\|h(u)\|= q(c)\}.
\end{array}
\end{equation*}
Namely, $X(c)$ is the set of minimizers of the augmented Lagrangian. 
\end{definition}

The next proposition will be used to justify the stopping criterion in the PDP algorithm.

\begin{proposition}[Search direction and stopping criterion for the PDP algorithm]\label{gradz0} \ \\
Assume that {\rm\hyperlink{H0}{(H0)}--\hyperlink{H2}{(H2)}} hold for Problem \eqref{eq:generalPr}. If $\hat{u} \in X(\hat{c})$, then the following facts hold.
\begin{itemize}
    \item[(i)] $\|h(\hat u)\| \in \partial q(\hat{c})$.
    \item[(ii)] If $h(\hat u)=0$ then $\hat u$ is an optimal primal solution, and $\hat c$ is an optimal dual solution. Conversely, assume $\hat c>0$ and that either $\hat u$ is an optimal primal solution, or $\hat c$ is an optimal dual solution. Then, we must have $h(\hat u)=0$. 
\end{itemize}
\end{proposition}
\begin{proof}
Even though the proof of (i) is standard and can be deduced from \cite[Proposition 3.1]{BIMinexact2013}, we include its proof here for convenience of the reader. To prove (i), use the definition of $q$ in \eqref{eq:dual_function} to write, for every $c\ge 0$,
\[
\begin{array}{rcl}
     q(c)&= & \inf_{u\in K} \varphi(u)+ c\|h(u)\| \le \varphi(\hat u)+ c\|h(\hat u)\| \\
        & & \\
      &&\ = \varphi(\hat u)+ \hat c\|h(\hat u)\| +(c- \hat c)\|h(\hat u)\|= q(\hat u)+(c- \hat c)\|h(\hat u)\|\,, 
      
      \end{array}
\]
where we used the fact that $\hat{u} \in X(\hat{c})$ and the definition of $q$ in the last equality. The above expression and Definition \ref{rSupDiff} yield $\|h(\hat u)\|\in \partial q(\hat{c})$, establishing (i). The first statement in part~(ii) has a proof similar to the one in \cite[Proposition 3.1]{BIMinexact2013} and hence omitted. So we prove the second statement in (ii). This statement automatically holds when $\hat u$ is an optimal primal solution because $\hat{u}\in S(P)$ and therefore it satisfies the equality constraints. We proceed now to prove the statemet when $\hat c$ is an optimal dual solution. Theorem \ref{th:StrongDual} and Remark \ref{rem:MP finite} yield
\begin{equation*}\label{eq:contradiction}
  \infty>M_P=M_D= q(\hat c)= \varphi(\hat u)+ \hat c \|h(\hat u)\|, 
\end{equation*}
where we used the assumption that $\hat{c}$ is a dual solution in the second equality, and the fact that $\hat{u} \in X(\hat{c})$ in the third one. Assume that  $h(\hat u)\not=0$. By Proposition \ref{prop:dualProperty}(ii), for every $\lambda>0$ we have that  $\hat c +\lambda\in S(D)$ and hence $q(\hat c +\lambda)=M_D$. We can write
\[
 \infty>M_P=M_D= q(\hat c+\lambda)= \varphi(\hat u)+ \hat c \|h(\hat u)\| + \lambda \|h(\hat u)\|.
\]
whose right-hand side  tends to infinity  for $\lambda\to+\infty$. This contradiction implies that $h(\hat u)=0$.
\end{proof}

\subsection{The Primal--Dual Penalty (PDP) algorithm}\label{sec:algorithmPDP}
Our Lagrangian is given by equation \eqref{eq:Lagrn} in Definition \ref{def:basic}, and it gives rise to a classical penalty method.  This motivates the name {\em primal--dual penalty (PDP) algorithm}, described below. We use in this algorithm the notation of Problem \eqref{eq:generalPr} and Definition \ref{def:basic}. By Remark \ref{rem: OCP examples}, we always have $h(u)=(h_1(u),\ldots,h_{m_1}(u))\in \R^{m_1}$ for some $m_1\in \mathbb{N}$. This allows us to consider finite dimensional norms for $h(u)$ in the definition of the PDP algorithm. Namely, we use the $\ell_1$ and the $\ell_{\infty}$ norms of $h(u)$. Our dual variable is $c_k\in \R_{+}$, while our primal variable is a function $u_k\in ({\mathcal L}^2([0,t_f];\R))^{m_2}$ for some $m_2\in \mathbb{N}$. 

\begin{algorithm}\label{alg:PDP}
\begin{quote}\rm
  
{\bf Primal--Dual Penalty (PDP) Algorithm} \\[2mm]
Let $\alpha,\varepsilon>0$.  Choose a sequence $\{\alpha_k\}\subset (0,\alpha)$.\\[1mm]
{\bf Step \boldmath{$0$}}. (Initialization) Choose $c_0>0$ and let $k:=0$.

\noindent {\bf Step \boldmath{$1$}}. (Solution of Subproblem and Stopping Criterion)
\begin{itemize}
    \item[(a)] Find $\ds u_k\in \argmin_{u\in U}{\ell(u,c_k)}$.
            \item[(b)] If $\| h(u_k)\|_\infty < \varepsilon$, stop.
\end{itemize}

\noindent {\bf Step \boldmath{$2$}}. (Selection of step-size and Update of Dual Variables)\\
Choose $s_k>0$ and set $\tilde{s}_k = (\alpha_k+1)s_k$. Update the penalty parameter by
\[
c_{k+1}:= c_k + \tilde{s}_k\,\|h(u_k)\|_1.
\]
Set $k:=k+1$, go to Step $1$.
  
\end{quote}
\end{algorithm}

\begin{remark}\label{increase}\rm
Proposition \ref{prop:dualProperty}(i) states that the dual function is non-increasing. Moreover, for $\{c_k\}$ generated by the PDP algorithm, strict increase of the sequence $\{q(c_k)\}$ is established in \cite[Theorem~3.1]{BIMinexact2013}). 
\end{remark}
\begin{remark}\label{rem:PDP vs Subgrad}\rm
By Proposition \ref{gradz0}(i), the search direction in Step 2 of the PDP algorithm is a classical subgradient direction for improving $q$. Proposition \ref{gradz0}(ii) justifies the stopping criterion in Step 1(b).
\end{remark}

Next we describe two choices for the step-size $s_k$ and the convergence results for each choice.

\subsection{Algorithm PDP-1}\label{ssec:alg1}
We consider in this section a step-size as in \cite[Algorithm 1]{BIMinexact2013}. Take two parameters $\beta>\eta>0$. Let $u_k$ be as Step 1(a). Consider the step-size
\begin{equation}\label{sk_DSG1}
s_k\in [\eta_k,\beta_k],
\end{equation}
where $\eta_k:=\min\{\eta,\|h(u_k)\|_2\}$ and $\beta_k:=\max\{\beta,\|h(u_k)\|_1+\|h(u_k)\|_2\}$, where $\|\cdot\|_2$ is the finite dimensional $\ell_2$ norm. With this choice of $s_k$, we denote the PDP~algorithm as PDP-1.
\begin{remark}\rm
Note that a constant step-size for all iterations is admissible.
\end{remark}

\noindent The next theorem states the convergence results for PDP-1. The proof of part (a) considers two possible cases, according to whether the dual sequence $\{c_k\}$ is bounded or not. The case of an unbounded sequence has a proof similar to \cite[Theorem 3.2]{BIMinexact2013}. The case of bounded dual sequence is slightly different, and can be found in \cite[Theorem 4.2]{BurachikXuemei2023}. The proof of part (b) follows directly from the fact that the dual sequence is strictly increasing. 

\begin{theorem}[Convergence of PDP-1]\label{primalCvg1}
Assume that $M_P=M_D$. Consider the primal sequence $\{u_k\}$ generated by {\rm PDP-1}. Take the parameter sequence $\{\alpha_k\}$ satisfying $\alpha_k\ge\bar \alpha$ for all $k$ and some $\bar \alpha>0$. The following hold.
\begin{itemize}
    \item[(a)] The primal sequence $\{u_k\}$ is bounded, all its weak accumulation points are primal solutions, and $\{q_k\}$ converges to the optimal value $M_P$.
    \item[(b)] If PDP-1 generates an infinite sequence $\{c_k\}$, then it converges if and only if it is bounded above, and in this case its supremum is a dual solution.
\end{itemize}
\end{theorem}

\subsection{Algorithm PDP-2}\label{sec:alg2}
In this section we consider the step-size proposed in \cite[Algorithm 2]{BIMinexact2013}, which ensures that the PDP algorithm converges in a finite number of steps. Take  $\beta>0$ and a sequence $\{\theta_k\} \subset\R_+$  such that
$\sum_j\theta_j=\infty$, and $\theta_k\leq\beta$ for all $k$.
Let $u_k$ be as Step 1(a). Consider the step-size
\begin{equation}\label{sk_DSG2}
s_k\in [\eta_k,\beta_k],
\end{equation}
where $\eta_k:= \theta_k/\|h(u_k)\|_1$ and $\beta_k:= \beta/\|h(u_k)\|_1$. 
With this choice of $s_k$, we denote the PDP algorithm as PDP-2. The following result from \cite[Theorem 4.5]{BurachikXuemei2023} states the convergence properties of PDP-2.

\begin{theorem}[Convergence of PDP-2]\label{primalCvg}
Assume that $M_P=M_D$. Let $\{u_k\}$ and $\{c_k\}$ be the sequences generated by {\rm PDP-2}. Suppose that the parameter sequence $\alpha_k\ge \bar \alpha>0$.  Then only one of the following cases occurs:\\
\noindent (a) There exists a $\bar{k}$
such that {\rm PDP-2} stops at iteration $\bar{k}$. As a
consequence $u_{\bar{k}}$ and $c_{\bar{k}}$ are
optimal primal and optimal dual
solutions, respectively. In this situation $\{c_k\}$ must be bounded.

\noindent (b) The dual sequence $\{c_k\}$ is unbounded. In this case, $\{q_k\}$ converges to $M_P$, and $\{u_k\}$ is bounded with all its weak accumulation points being primal solutions.
\end{theorem}

\section{A class of optimal control problems} \label{sec:problemP1}
\label{sec:generalOCP}

\subsection{Problem Formulation}\label{ssec:OCP formulation}
In later sections, we will use the PDP algorithm to solve the optimal control of the constrained double integrator and of the free-flying robot. These problems fall into the class of optimal control problems described in Problem~$(P1)$ below. 
Let the space ${\mathcal L}^2([0,t_f];\R^m)$ be as in Definition \ref{def:LP} for $p=2$. Consider also the Sobolev space ${\mathcal W}^{1,2}([0,t_f];\R^n)$ as in the same definition. We consider the following class of optimal control problems.
\[
(P1)\,
  \left\{
    \begin{array}{rlll}
      \min &\ds{\frac{1}{2}\int_0^{t_f}} & f_0(u(t))\,dt\\[4mm]
      {\rm subject \, to}&& \dot x(t) =  f(x(t),u(t)),\\[2mm]
        && x(0)=x_0\,,\, x(t_f)=x_f,\\[2mm]
       && \mbox{and} \,\, |u_i(t)|\le a_i, \, i=1,\ldots,m,
    \end{array}
  \right.
\]
where the state variable $x(t)=(x_1(t),\ldots,x_n(t))\in \R^n$, $x\in {\mathcal W}^{1,2}([0,t_f];\R^n)$,  the control variable $u(t)=(u_1(t),\ldots,u_m(t))\in \R^m$, $u\in {\mathcal L}^2([0,t_f];\R^m)$,  $a_i>0$, for $i=1,\ldots,m$. Let $f$ be linear in $u$, and
$f_0:\R^m\to \R$ and $f:\R^n\times\R^m\to \R^n$ be $C^1$ in their arguments. We assume that $a_i$ is large enough so as to ensure that Problem $(P1)$ has solutions.

\subsection{Optimality conditions}
\label{sec:optimalityCondn}

We now derive the first-order necessary conditions of optimality for the optimal control Problem~$(P1)$ by means of the maximum principle \cite[Theorem~7.2]{Hestenes1966}.  The Hamiltonian function $H:\R^n \times \R^m \times \R^n \times \R \to \R$ for Problem~$(P1)$ is defined in the usual way as
\begin{equation} \label{Ham}
H(x,u,\lambda,\lambda_0) := \lambda_0\,f_0(u) + \lambda^Tf(x,u),
\end{equation}
where $\lambda_0\in \R$ and the adjoint (or costate) variable vector
$\lambda(t) := (\lambda_1(t),\ldots,\lambda_n(t))\in \R^n$.  We further note that $\lambda\in {\mathcal W}^{1,2}([0,t_f];\R^n)$.  In~\eqref{Ham}, we have dropped the dependence on $t$ of the variables, for clarity in appearance.  Also keeping up with the tradition we define
\[
H[t] := H(x(t),u(t),\lambda(t),\lambda_0)\,.
\]
Next, we assume that the adjoint variable vector satisfies the differential equation
\begin{equation} \label{eq:adjoint}
\dot{\lambda}(t) := -H_x[t]\,,
\end{equation}
where $H_x = \partial H / \partial x$. Suppose that $(x,u)\in {\mathcal W}^{1,2}([0,t_f];\R^n) \times {\mathcal L}^2([0,t_f];\R^m)$ is an optimal pair for Problem~$(P1)$.  Then -- see~\cite{Hestenes1966} -- there exist $\lambda_0 \ge 0$ and a continuously differentiable adjoint variable vector $\lambda$ as defined in~\eqref{eq:adjoint}, such that $\lambda(t)\neq0$ for all $t\in[0,t_f]$, and that, for all $t\in[0,t_f]$,
\begin{equation}\label{eq:controlOptCdn}
u_i(t) = \Argmin_{w\in[-a_i,a_i]} H(x(t),(u_1(t),\ldots,w,\ldots,u_m(t)),\lambda(t),\lambda_0)\,,\ \ \mbox{for } i=1,\ldots m\,.
\end{equation}
Note that in~\eqref{eq:controlOptCdn}, the minimization is carried out with respect to $w$ only, which replaces $u_i(t)$ in the $i$th position of the $u(t)$ vector. Problem~($P1$) is said to be {\em normal} when $\lambda_0 >0$.  When $\lambda_0 = 0$, the maximum principle does not convey sufficient information and Problem~($P1$) and its solution are referred to as {\em abnormal}.
In~\cite[Example 2, pp. 2800--2801]{KayNoa2008} it is shown that if the control system involving linear state ODEs is controllable then $\lambda_0\neq0$, i.e., the optimal control problem is normal.  From now on, we assume that Problem~($P1$) is normal, and set $\lambda_0 = 1$, without loss of generality.

Let $f_0(u(t)):= \left( u_1^2(t) +\dots +u_m^2(t)\right)/2$.  Incorporating this special form of $f_0$ and the linearity of $f$ in $u$, \eqref{eq:controlOptCdn} reduces to
\begin{equation}\label{eq:controlOpt}
u_i(t)=\,
  \left\{
    \begin{array}{lcl}
    -\ds\sum_{j=1}^n\lambda_j(t)f_{u_i}(x(t)) & , & \mbox{if}\,\, -a_i<\ds\sum_{j=1}^n\lambda_j(t)f_{u_i}(x(t))< a_i,\\[2mm]
    \ \ \,a_i & , & \mbox{if}\,\, \ds\sum_{j=1}^n\lambda_j(t)f_{u_i}(x(t)) \leq -a_i,\\[2mm]
    -a_i & , & \mbox{if}\,\, \ds\sum_{j=1}^n\lambda_j(t)f_{u_i}(x(t)) \geq a_i,\\
    \end{array}
  \right.
\end{equation}\\
for $i=1,\ldots ,m$. Due to the linearity of $f$ in $u$, the $u_k$s, $k = 1,\ldots,m$, do not appear explicitly in $f_{u_i} = \partial f/\partial u_i$\,. For this reason, we write $f_{u_i}(x(t))$ for $f_{u_i}(x(t),u(t))$. 
For the case when the $i$th control variable $u_i(t)$ is not constrained, i.e., $a_i = \infty$, the expression in~\eqref{eq:controlOpt} reduces to
\[
u_i(t) = -\sum_{j=1}^n\lambda_j(t)f_{u_i}(x(t))\,.
\]

\subsection{Direct discretization and the settings for computations}\label{sec:discritization}

For computations, we discretize the problem by using the following notation. Suppose that the optimal control problem has $n$ states and $m$ control variables.  We consider discrete approximations of the functions over the partition $0 = t_0 < t_1 < \ldots < t_N = t_f$ such that
\[
t_{i+1} = t_i + \Delta t\,,\ \ i = 0,1,\ldots,N\,,
\]
$\Delta t := t_f/N$ and $N$ is the number of subdivisions. Let $u_{rj}$ be an approximation of $u_r(t_j)$, i.e., $u_{rj}\approx u_r(t_j)$, $r = 1,\ldots,m,\ j = 0,1,\ldots,N-1$; similarly, $x_{pi}\approx x_p(t_i)$, $p = 1,\ldots,n,\ i = 0,1,\ldots,N$.

We use the optimization modelling language AMPL~\cite{AMPL} in coding for solving our optimal control problems and get the discrete (finite-dimensional) solution. We employ the optimization software Ipopt~\cite{WacBie2006} (version 3.12.13) for solving the subproblems in the PDP algorithm, i.e. minimize the augmented Lagrangian in Step 1(a) of Algorithm \ref{alg:PDP}. We also solve the same optimal control
problems by using Ipopt on its own, in order to make comparisons with our PDP algorithm. 

The AMPL–Ipopt suite was run on a Dell desktop, with the operating system Windows 10 Enterprise (version 20H2), the processor 2.40 GHz\, Intel Core i7 and the memory 16 GB\, 2666 MHz\, SODIMM. We have used the Ipopt options \texttt{max\_iter=1000}, \texttt{tol=1e-8} and \texttt{acceptable\_tol=1e-8}.

AMPL can also be paired with other optimization software, e.g., Knitro~\cite{Knitro}, SNOPT~\cite{GilMurSau2005} and TANGO~\cite{AndBirMarSch2007,BirMar2014}, in solving the subproblems in the PDP algorithm.

In the next two sections, we will use the PDP algorithm to solve two optimal control problems and carry out numerical experiments with $s_k$ chosen as in~\eqref{sk_DSG1} and \eqref{sk_DSG2}. 

\section{Application to the Constrained Double Integrator}\label{sec:DoubleIntegrator}
In this section, we use the PDP algorithm to solve the optimal control of the constrained double integrator of a car (Problem $(P2)$ below). We proceed to describe the classical mathematical model, and then we will present a reformulation that fits the format of Theorem~\ref{th:example_verify_H}. Suppose that, at time $t$, the position of a car modelled as a point mass travelling on a flat surface is given by $y(t)$.  Then its velocity and acceleration are $\dot y(t) = (dy/dt)(t)$, and $\ddot{y}(t) = (d^2 y/dt^2)(t)$, respectively.  Suppose that the summation of all the external forces applied to the car is $u(t)$. Then  by Newton's second law of motion $\ddot y(t) = u(t)$ (assuming unit mass).
Let $x_1 := y$ and $x_2 := \dot y$. We impose a constraint on $u$ that $|u(t)|\le a$. We aim to minimize the squared ${\mathcal L}^2$-norm of the acceleration, with starting position and velocity $x_1(0) = s_0$ and $x_2(0) = v_0$, and final position and velocity $x_1(1) = s_f$ and $x_2(1) = v_f$, within one unit of time. This problem can then be mathematically modelled as follows.
\[
(P2)\,
  \left\{
    \begin{array}{rllll}
      \min &\ds{\frac{1}{2}\int_0^1 u^2(t)dt} & & & \\[4mm]
      {\rm subject \, to} &\dot x_1(t)=x_2(t), & x_1(0)=s_0, &x_1(1)=s_f, \\[1mm]
       &\dot x_2(t)=u(t), & x_2(0)=v_0, &x_2(1)=v_f, & |u(t)|\le a,\, \forall\, t\in [0,1].  \\
    \end{array}
  \right.
\]
Here the position $x_1$ and the velocity $x_2$ are the state variables. Assuming that we can change $u$ the way we like, it is nothing but the control variable of the problem. Due to the box constraint on the control variable, an analytical solution for Problem $(P2)$ is in general not possible.  

\subsection{Problem $(P2)$ verifies (H0)--(H2)}
The box constraints on $u$ can be written as
$u\in K_1$, where $K_1:=\{v\in {\mathcal L}^2([0,1],\R)\::\: |v(t)|\le a,\, \forall\, t\in [0,1]\}$. To  use Theorem~\ref{th:example_verify_H}, a first step is to show that the ODE system in  Problem $(P2)$ can be equivalently written as an equality constraint of the form $h(u)=0$, for a suitable function $h$.  We do this in the following lemma.

\begin{lemma}[ODEs as equality constraints]\label{lem:P2-h=0}
Consider the ODE system 
\[
(S1)\,
  \left\{
    \begin{array}{lll}
     \dot x_1(t)=x_2(t)\,, & x_1(0)=s_0\,, &x_1(1)=s_f\,, \\[2mm]
    \dot x_2(t)=u(t)\,, & x_2(0)=v_0\,, &x_2(1)=v_f\,.
    \end{array}
  \right.
\]
Define 
\begin{equation}\label{eq:r1_r2}
    r_1:= v_0+s_0-s_f\quad \hbox{and}\quad r_2:= v_0-v_f\,.
\end{equation}
The system $(S1)$ can be written as $h(u)=0$, where $h:{\mathcal L}^2([0,1];\R)\to \R^2$ is defined as
 \begin{equation}\label{eq:h_p1}
 h(u):=  \left[ 
 \begin{array}{l}
     h_1(u)  \\[2mm]
     h_2(u) \\
 \end{array}
 \right] := \left[ 
 \begin{array}{l}
     r_1+\ds\int_0^1\left[ \int_0^\tau u(s)ds \right]d\tau   \\[4mm]
      r_2+ \ds\int_0^1 u(\tau)d\tau\\
 \end{array}
 \right].
 \end{equation}
 \end{lemma}
 \begin{proof}
 Using the ODE constraints we re-write Problem $(S1)$ as follows:\\[2mm] 
\[
(S1')\,
  \left\{
    \begin{array}{r}
      x_1(1)-s_f=0\,, \\[2mm]
       x_2(1)-v_f=0\,, \\
    \end{array}
  \right.
\]
\\[2mm]
where $x_i(1)$ for $i=1,2$ are defined as follows
\begin{equation} \label{eq:xP1reformulation}
\begin{array}{ll}
& \ds x_2(t):=v_0+ \int_0^t u(\tau)d\tau, \\[5mm]
& \ds x_1(t):=s_0+\int_0^t x_2(\tau)d\tau
 =s_0+\int_0^t\left[v_0+\int_0^\tau u(s)ds\right]d\tau.
\end{array}
\end{equation}
Note that the right hand sides in \eqref{eq:xP1reformulation} are affine functions of $u$. Using this definition and \eqref{eq:r1_r2}, it is direct to check that  $h(u)=0$ if and only if $x_1(1)=s_f$ and $x_2(1)=v_f.$ 
\end{proof}

\begin{theorem}[Problem $(P2)$ verifies (H0)--(H2)]
\label{th:verify_p1}
Let $h$ be defined as in \eqref{eq:h_p1}. Consider for Problem $(P2)$ the dualizing parameterization $f:{\mathcal L}^2([0,1];\R)\times \R^2\to \R_{\infty}$ defined by
\[
f(u,z):=\varphi(u)+\delta_{z}(h(u))+\delta_{K_1}(u).
\]  
where $\varphi(u):={\ds\frac{1}{2}\int_0^1 u^2(t)dt}$, and $K_1=\{u\in {\mathcal L}^2[0,1]\::\: |u(t)|\le a,\, \forall\, t\in [0,1]\}$. Then assumptions {\rm\hyperlink{H0}{(H0)}--\hyperlink{H2}{(H2)}} hold for Problem $(P1)$.
\end{theorem}
\begin{proof}
It is clear that $\varphi$ is proper and coercive. By Theorem \ref{th:BA2} it is proper, coercive and w-lsc. This fact, together with the definition of $f$, imply that assumptions (a)
and (c) from Theorem \ref{th:example_verify_H} hold. To complete the proof, we need to check that assumption (b) from Theorem \ref{th:example_verify_H} holds. Namely, we need to show that the set $K_1\cap h^{-1}(z) $ is w-closed. In fact, we will show that this set is w-compact, and this will provide the desired weak closedness by Fact \ref{lem:Weakly Closed}(b). Using Theorem \ref{th:ES}, it is enough to show that $K_1\cap h^{-1}(z) = \left(K_1\cap h_1^{-1}(z_1)\right) \cap\left( K_1\cap h_2^{-1}(z_2)\right) $ is sequentially weakly compact. Define $\Gamma_j:=K_1\cap h_j^{-1}(z_j) $ for $j=1,2$. We will show that each $\Gamma_j$ is sequentially w-compact. Fix $j\in \{1,2\}$. Take a sequence $\{u_k\}\subset \Gamma_j$. Since $\Gamma_j\subset K_1$ and $K_1$ is weakly compact, there exists a subsequence 
$\{u_{k_l}\}\subset \{u_k\}$  s.t. $u_{k_l}\rightharpoonup u\in K_1$.  
Using \eqref{eq:h_p1} and the notation of Lemma \ref{lem:varfi} 
with $T:=1$ and $m=1$, we have that
\begin{equation}\label{eq:eta-ro2}
 \begin{array}{rcll}
   h_1(u)  &=  &  r_1 +\rho_1^{\varphi_1}(u,1),
    
   &h_2(u)  = r_1 + \eta_1^{\varphi_1}(u,1),
\end{array}
\end{equation}
where $\varphi_1(s)=1$ for every $s\in [0,1]$. Because  $\{u_{k_l}\}\subset \Gamma_j$ we have that $h_j(u_{k_l})=z_j$. By Lemma \ref{lem:varfi}(c) and \eqref{eq:eta-ro2} we deduce that 
\[
z_j= \lim_{l\to \infty} h_j(u_{k_l})= h_j(u),\, j=1,2.
\]
Hence, $u\in \Gamma_j$ for $j=1,2$. This shows that both $\Gamma_1$ and $\Gamma_2$ are sequentially w-compact and thus the set $K_1\cap h^{-1}(z) = \Gamma_1\cap \Gamma_2$ is sequentially weakly compact. By Theorem \ref{th:ES}, it is w-compact and therefore w-closed. This completes the proof.
\end{proof}

\subsection{Numerical solution of Problem $(P2)$}
Using Equation \eqref{eq:adjoint}, the adjoint variables for this problem can simply be written as
\begin{equation*}\label{AdjointCarConstr}
    \lambda_1(t)=c_1 \quad \mbox{and}\quad
    \lambda_2(t)=-c_1t-c_2,
\end{equation*}
for all $t\in[0,1]$. Here $c_1$ and $c_2$ are real constants. Using \eqref{eq:controlOpt}, the optimal control for this problem is given by
\begin{equation}\label{controlCarConstr}
u(t)=\,
  \left\{
    \begin{array}{cl}
    -\lambda_2(t), & \mbox{if}\,\, -a\leq\lambda_2(t)\leq a\,,\\[1mm]
    a, & \mbox{if}\,\, \lambda_2(t)\leq -a\,,\\[1mm]
    -a, & \mbox{if}\,\, \lambda_2(t)\geq a\,,
    \end{array}
  \right.
\end{equation}
for all $t\in[0,1]$.
We take $a=2.5$, $s_0=0$, $s_f=0$, $v_0=1$, and $v_f=0$ in our numerical implementation. We discretize Problem ($P2$) over 1000 time partition points and use the PDP algorithm under both step-size of type 1 as in \eqref{sk_DSG1} and step-size of type 2 as in \eqref{sk_DSG2} to solve it. 

The feasibility tolerance $\varepsilon$ in Step 1(b) of Algorithm \ref{alg:PDP} is set at $10^{-6}$. The software package Ipopt is employed in solving the sub-problem, namely, in finding the minimizer of the Lagrangian in each iteration. We assign the parameters for the step-sizes of types 1 and 2 as follows.
\begin{itemize}
    \item step-size of type 1: $\alpha_k=1$, $\eta_k=0.1$ and $\beta_k=1$ for all $k$ and  $s_k$ is taken to be the midpoint of $[\eta_k,\beta_k]$. 
    \item step-size of type 2: $\alpha_k=1$ and $\theta_k=1$  for all $k$, $\beta=3$. Using the definition of step-size of type 2 that $\eta_k= \theta_k/\|h(u_k)\|_1$ and $\beta_k= \beta/\|h(u_k)\|_1$, we have $s_k\in [\eta_k,\beta_k]$ obtained as $s_k\in \left[\dfrac{1}{\|h(u_k)\|_1},\dfrac{3}{\|h(u_k)\|_1}\right]$. Since $c_{k+1} = c_k + (\alpha_k+1)s_k\|h(u_k)\|_1$ by the PDP algorithm, combining the range of $s_k$ and $\alpha_k$, we have that
    \[
    c_{k+1}- c_k = (\alpha_k+1)s_k\|h(u_k)\|_1 \in (\alpha_k+1)[1,\,3]=[2,\,6].
    \]
   That is to say, the increment of $ c_{k+1}- c_k$ is a quantity in the range of $[2,\,6]$.
\end{itemize}
In our experiments, PDP-2 uses step-size of type 2 with the parameters above, and has usually found the solution in four or five iterations.  The numerical results obtained by the PDP algorithm are shown in Figure~\ref{fig:carConstr}.  One should note that the first-order optimality of the control variable in Figure~\ref{fig:car_u} (as the necessary condition) is certified by the adjoint variable $\lambda_2$ in Figure~\ref{fig:car_lambda} via the expression in~\eqref{controlCarConstr}.  We include the graphs of the dual function and the dual iterates by PDP-1 and PDP-2 in Figure \ref{fig:dualConstr}.

We plot the function iterates $u_k$ in Figure~\ref{fig:uCvgEg2}, where $u_k$ are the minimizers of the dual function $q(c_k)$ (for $k=0,1,2,3,4$). 

We use different number of discretization points $N$ to compute $u(t)$ for Problem $(P2)$ by our PDP algorithm. We plot the solution of $u(t)$ obtained by the PDP algorithm with $N=20, 100$ and $\infty$ in Figure~\ref{fig:uVSnEg2} and for comparison, plot the solution of $u(t)$ by using Ipopt alone in Figure~\ref{fig:uVSnEg2_Ipopt}. When $n\le 19$, Ipopt fails to find a solution for Problem $(P2)$, while the PDP algorithm gives a solution when $N$ is as small as $10$.

\begin{figure}[t!]
$\begin{array}{lr}
  \begin{subfigure}{.49\linewidth}
  \includegraphics[width=\linewidth]{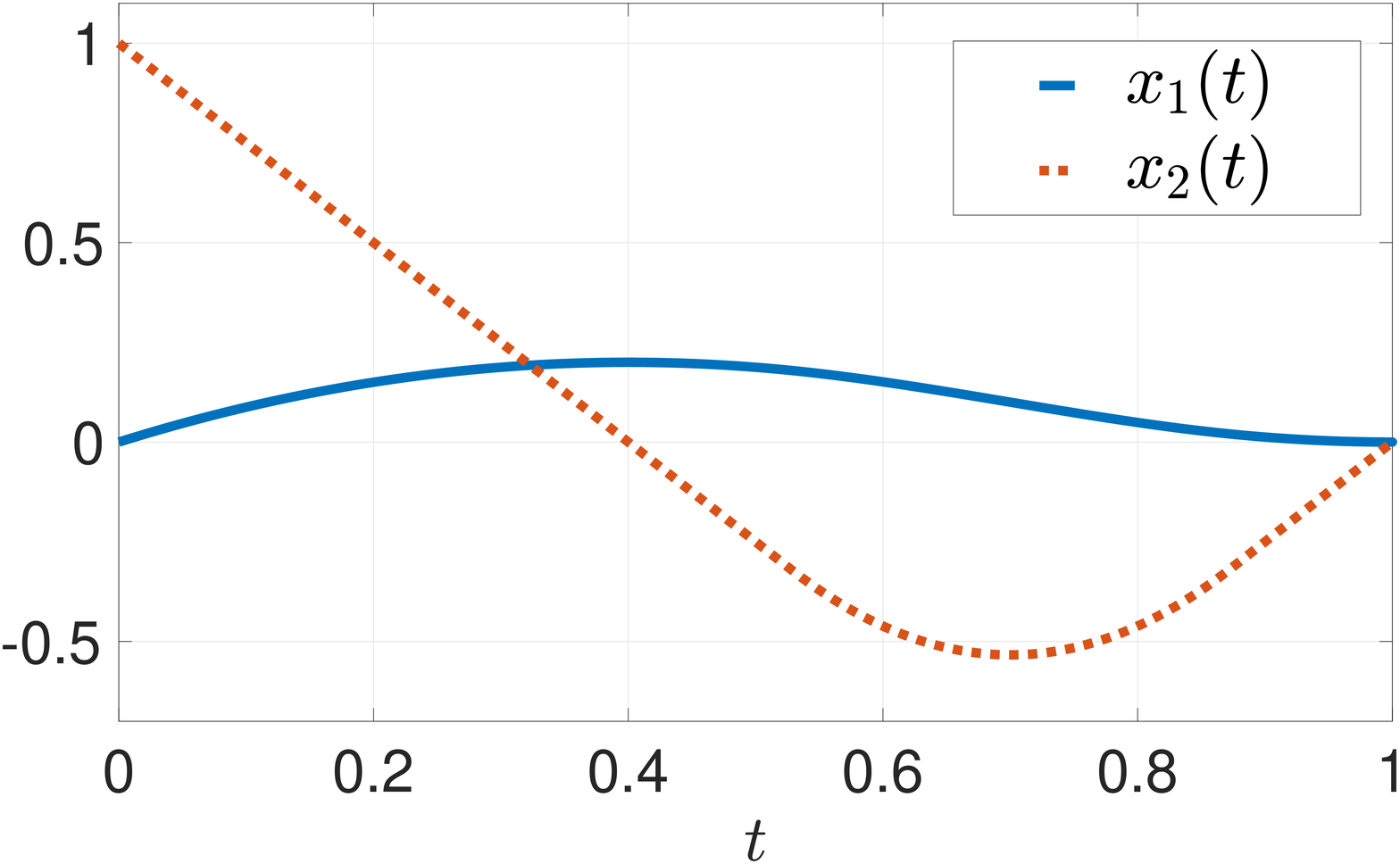}
  \caption{\small Optimal state variables.}
  \label{fig:car_x}
\end{subfigure}
&
  \begin{subfigure}{.49\linewidth}
\includegraphics[width=\linewidth]{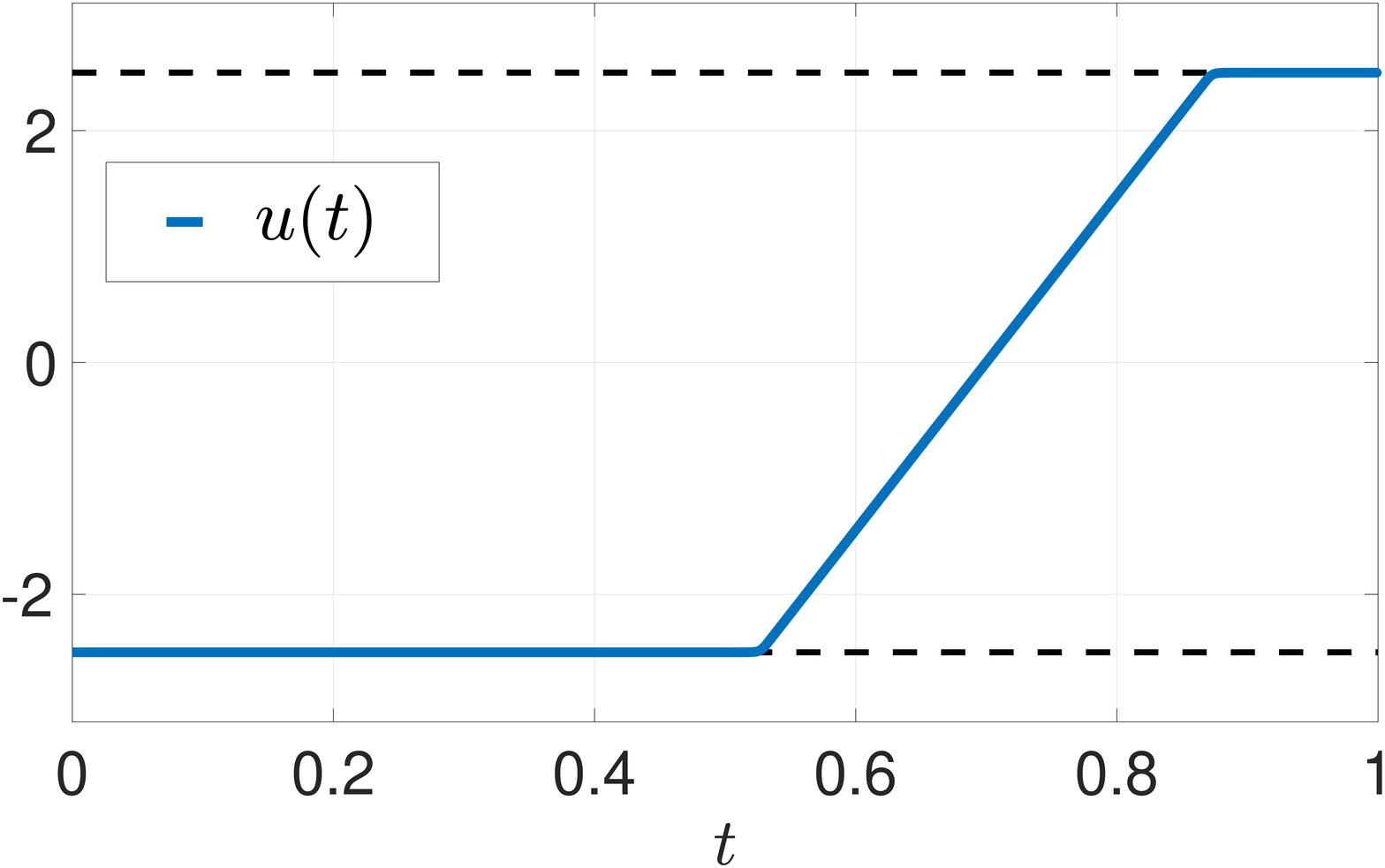}
  \caption{\small Optimal control variable.}
  \label{fig:car_u}
  \end{subfigure} \\
&
  \begin{subfigure}{.49\linewidth}
   \includegraphics[width=\linewidth]{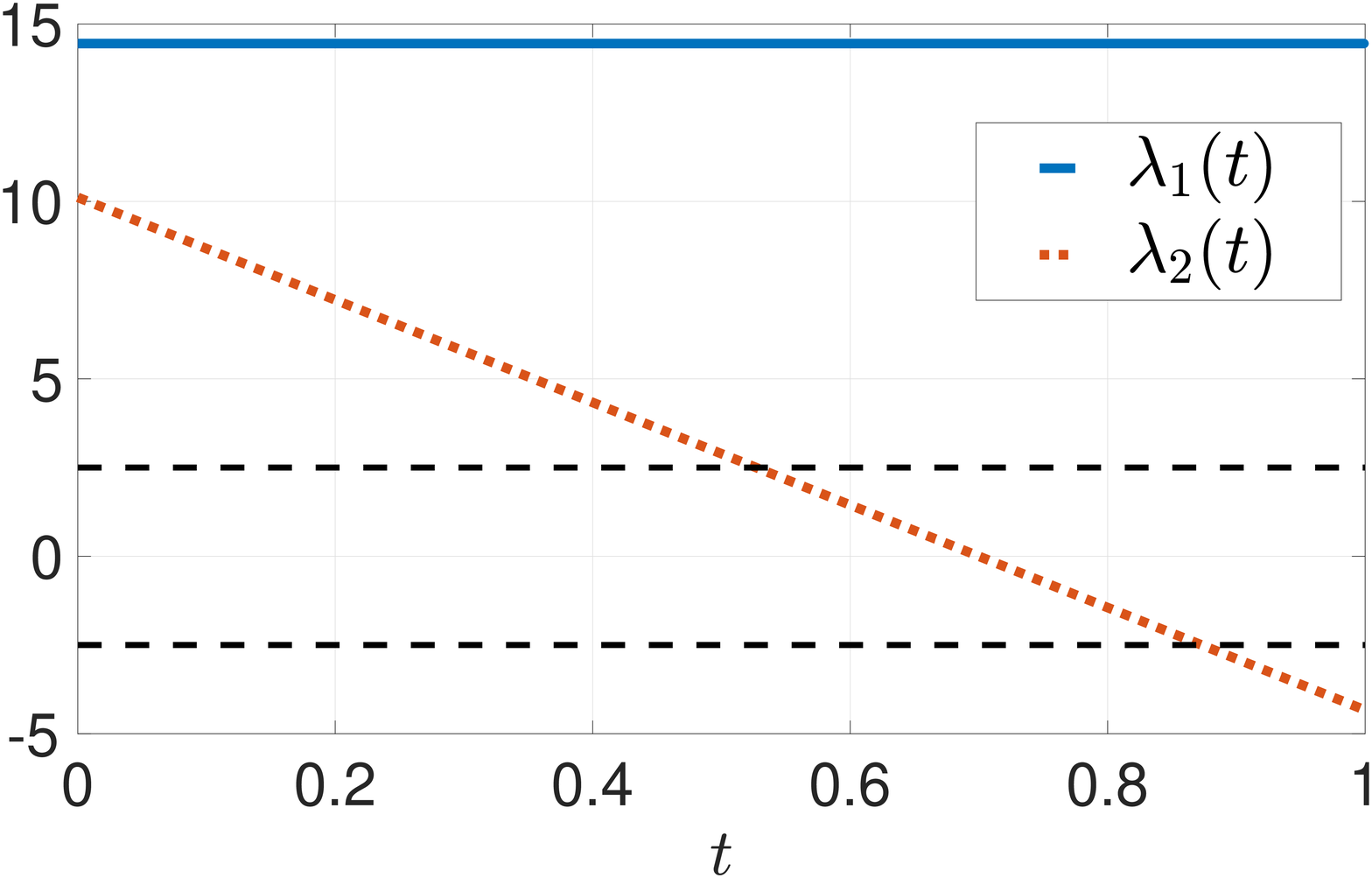}
   \caption{\small Adjoint variables.}
   \label{fig:car_lambda}
\end{subfigure}

\end{array}$
\caption{\sf Problem~$(P2)$---Solutions as obtained by the PDP algorithm.}
  \label{fig:carConstr}
  \end{figure}

 \begin{figure}[t!]
  \centering
  \begin{subfigure}{.49\linewidth}
\includegraphics[width=\linewidth]{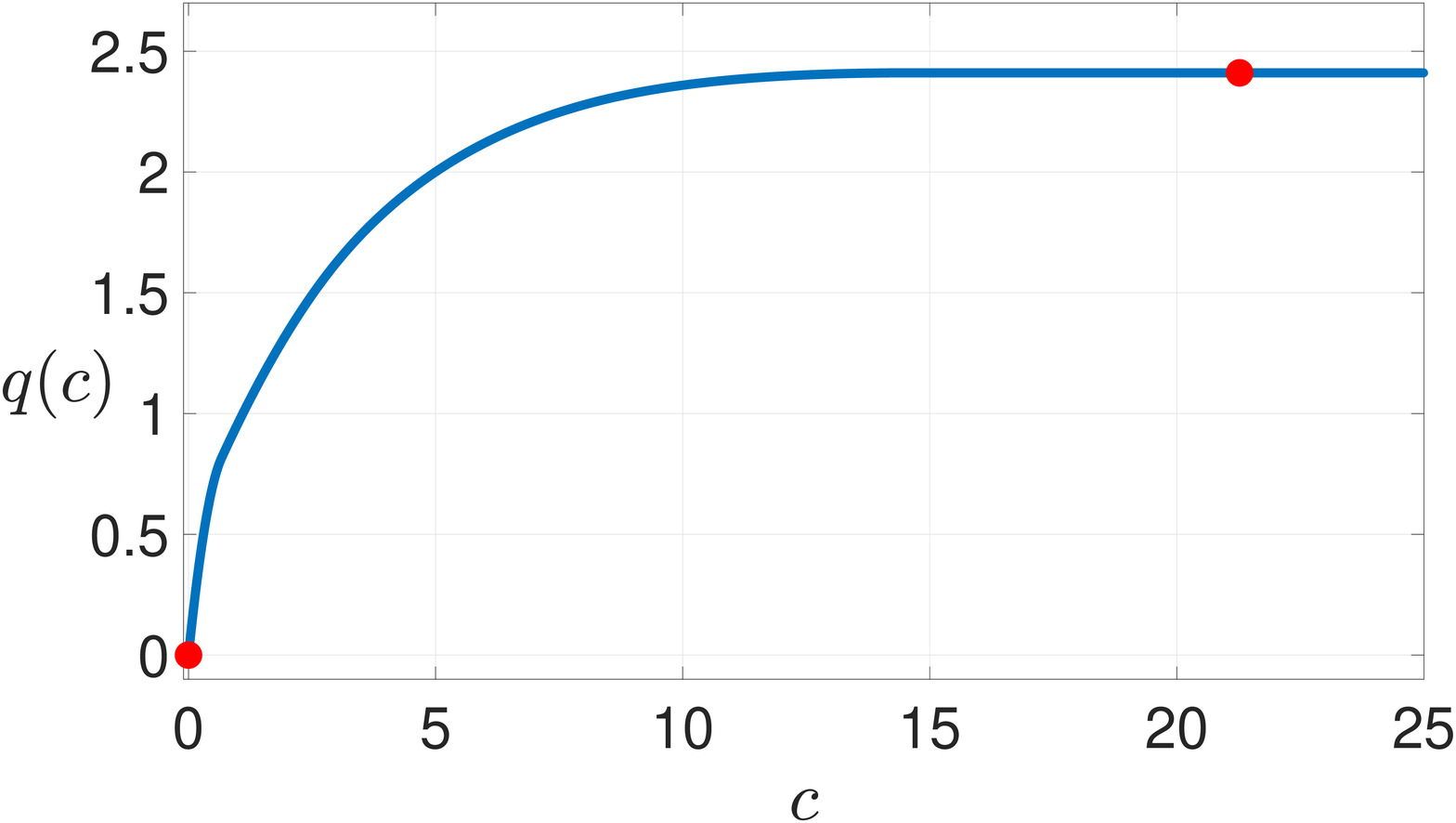}
  \caption{\small Iterations with $s_k$ in \eqref{sk_DSG1}.}
  \label{fig:dualPDP1}
\end{subfigure}
\begin{subfigure}{.49\linewidth}
\includegraphics[width=\linewidth]{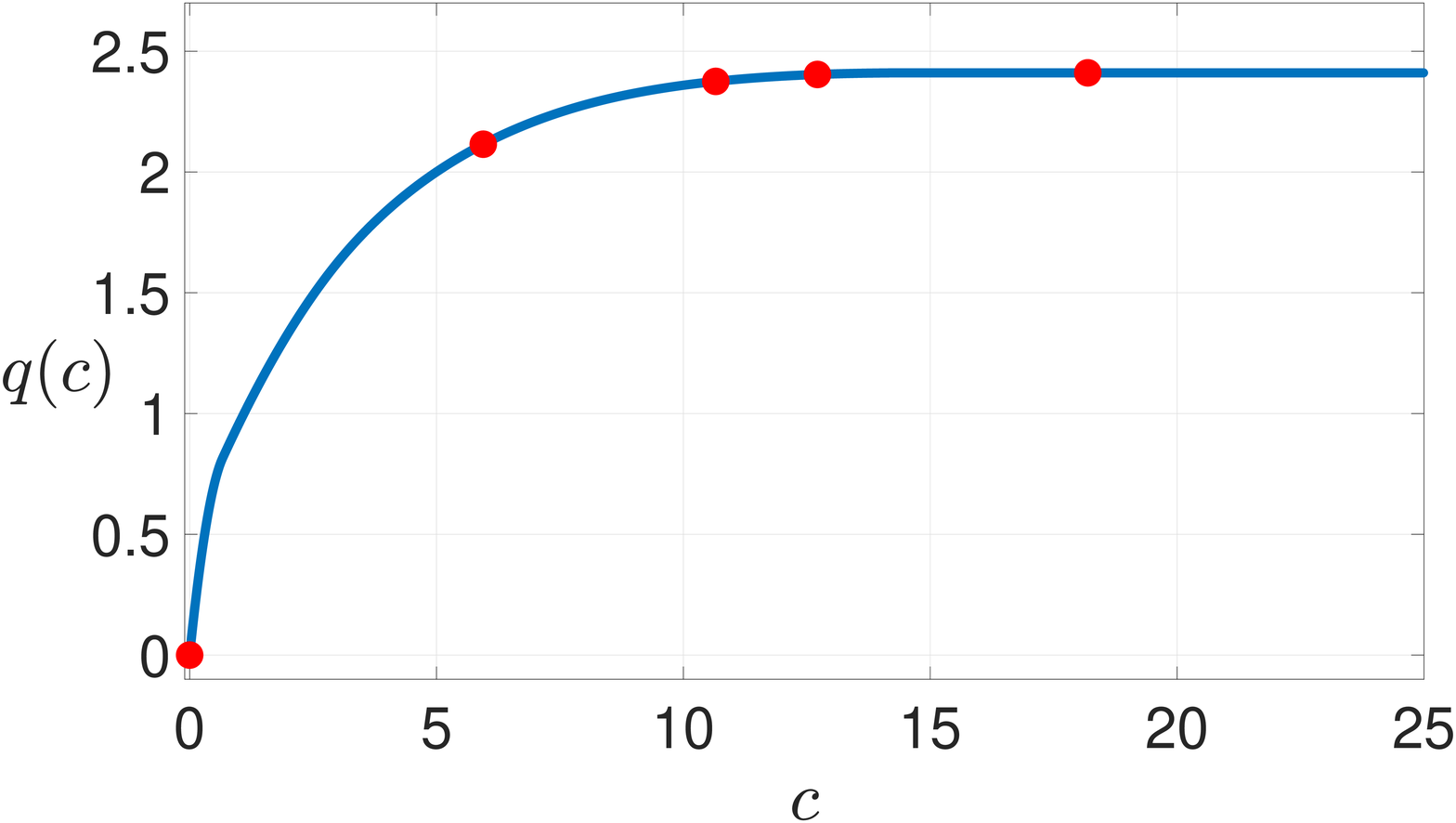}
  \caption{\small Iterations with $s_k$ in \eqref{sk_DSG2}.}
  \label{fig:dualPDP2}
  \end{subfigure}
  \caption{\sf Problem $(P2)$---The dual function updates (shown by red dots on the blue curve representing the graph of the dual function) in each iteration of the PDP algorithm using step-sizes of type 1 and 2.}
  \label{fig:dualConstr}
\end{figure}

\begin{figure}[t!]
  \centering
    \includegraphics[width=0.7\textwidth]{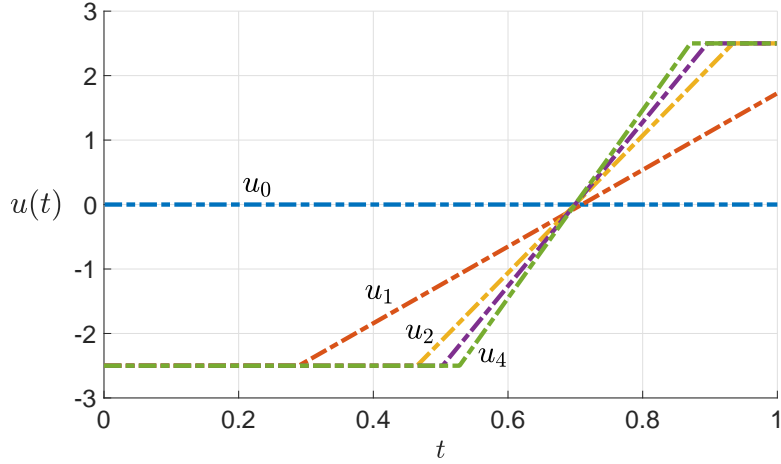}
    \caption{\sf Problem $(P2)$---The iterations ($u_0$, $u_1$, $u_2$, $u_3$ and $u_4$) of $u(\cdot)$ under the PDP algorithm with step-size of type 2. The iterate $u_3$ is (indicated in purple but) not labelled for clarity.}
  \label{fig:uCvgEg2}
\end{figure}

 \begin{figure}[t!]
  \centering
  \begin{subfigure}{.49\linewidth}
\includegraphics[width=\linewidth]{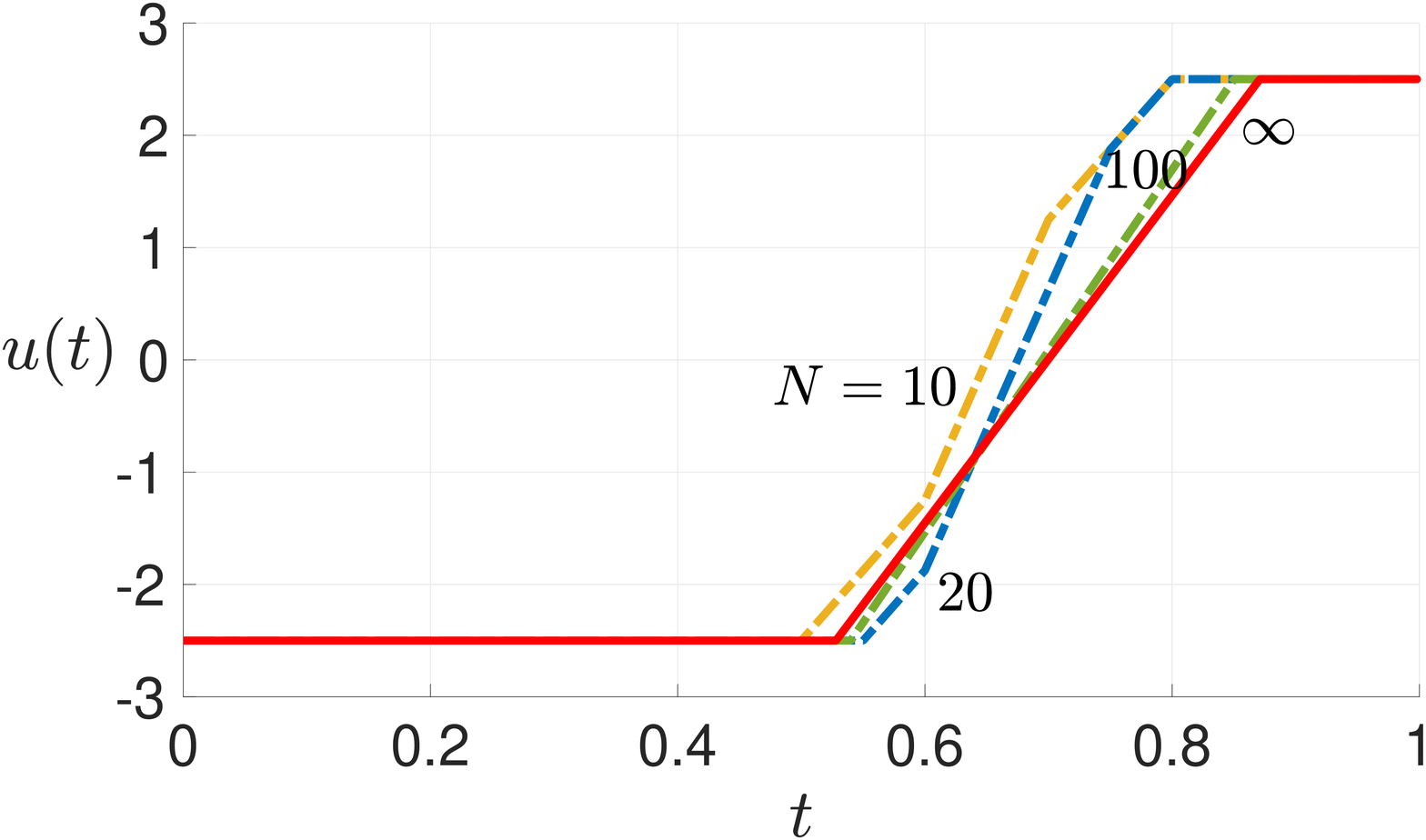}
  \caption{\small $u(\cdot)$ obtained by the PDP algorithm.}
  \label{fig:uVSnEg2}
\end{subfigure}
\begin{subfigure}{.49\linewidth}
\includegraphics[width=\linewidth]{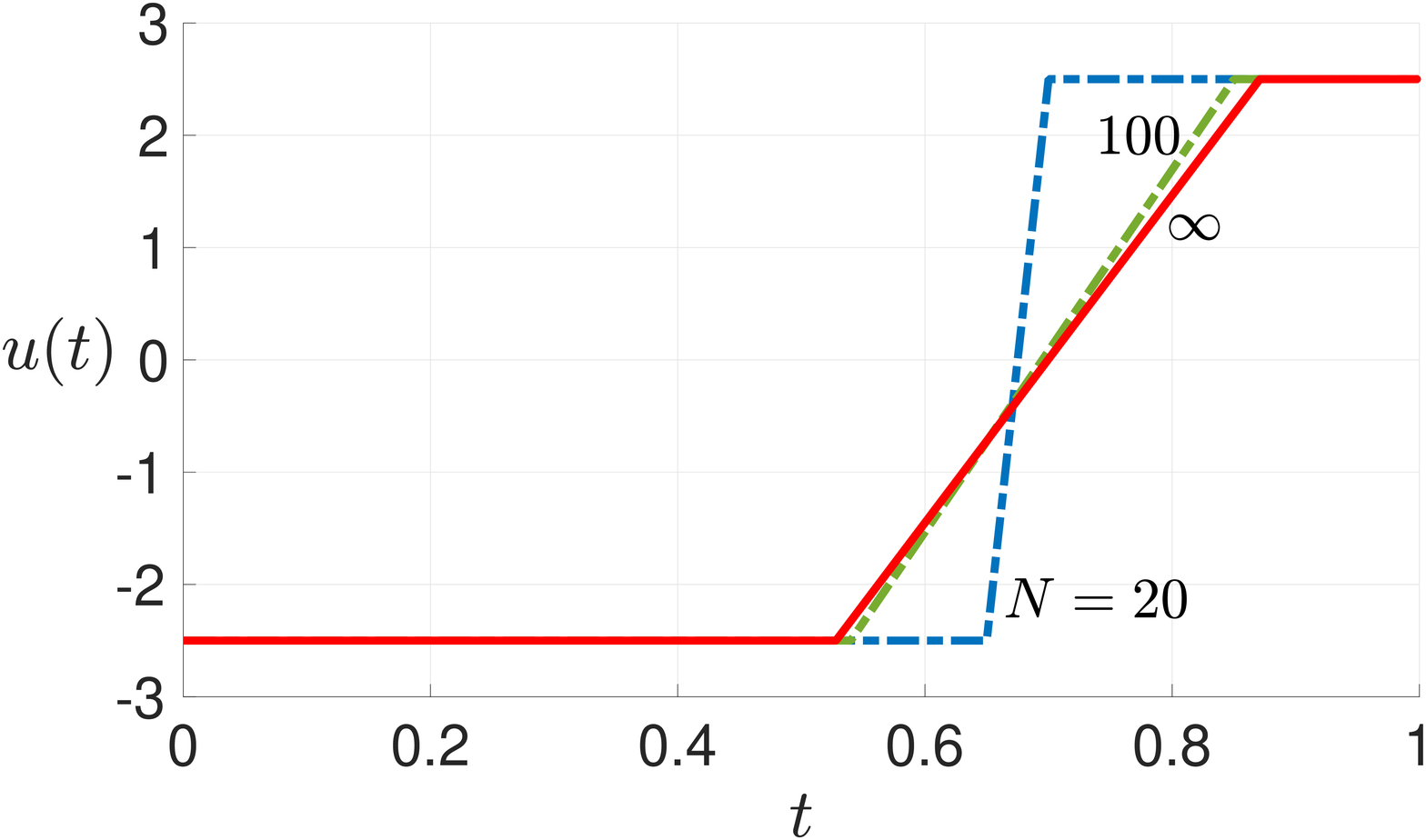}
  \caption{\small $u(\cdot)$ obtained by Ipopt alone.}
  \label{fig:uVSnEg2_Ipopt}
  \end{subfigure}
  \caption{\sf Problem $(P2)$---Solution $u(\cdot)$ obtained by the PDP algorithm with step-size of type 2 and by Ipopt alone, under different number of discretization points ($N=10, 20$ and $100$; the solution when $N\to\infty$ is represented in red solid-line curve by the solution obtained with $N=1000$).}
  \label{fig:Eg2_uVSn}
\end{figure}

We have also used Ipopt on its own to solve the discretization of Problem~$(P2)$ (not as a part of the PDP algorithm). These two methods achieve both $100\%$ success rate in finding the solution, and Ipopt alone uses less CPU time than the PDP algorithm. We note that a much more efficient method using projection techniques is provided by Bauschke, Burachik and Kaya in~\cite{BBY} for a class of convex optimal control problems, including Problem~$(P2)$.  Therefore, neither Ipopt nor the PDP algorithm should be the method of choice for Problem~$(P2)$.

\section{Application to the Free Flying Robot Problem}
\label{sec:FFRobot}

The PDP algorithm can also solve non-convex problems, including the challenging Problem~$(P3)$ below, involving the so-called free-flying robot (FFR).  Problem~$(P3)$ is highly non-convex and thus cannot be solved by existing projection methods.  This warrants implementing our PDP algorithm for solving it and comparing it with the approach using Ipopt on its own.

\subsection{The mathematical model for Problem $(P3)$}
The mathematical model for this problem is as follows. The aim is to minimize the fuel consumption of a robot which is moving at a constant height from an initial to a final equilibrium position. The robot can be controlled by the thrust of two jets. We use $x_1$ and $x_2$ for the coordinates of the FFR, $x_4$ and $x_5$ for the corresponding velocities, $x_3$ for the direction of thrust, $x_6$ for the angular velocity, and $u_1$ and $u_2$ for the thrusts of the two jets. The model was formulated initially in \cite{Sakawa1999} and further studied in ~\cite{BY2013,BYir,VosMau2006}.  We use the control constraints as in ~\cite{BY2013,BYir,VosMau2006}. Figures~\ref{fig:ffRobot} and \ref{fig:trjectory} respectively show a diagrammatic illustration of the model and the solution trajectory.

\begin{figure}[t]
  \begin{subfigure}{.45\linewidth}
\includegraphics[width=1\linewidth]{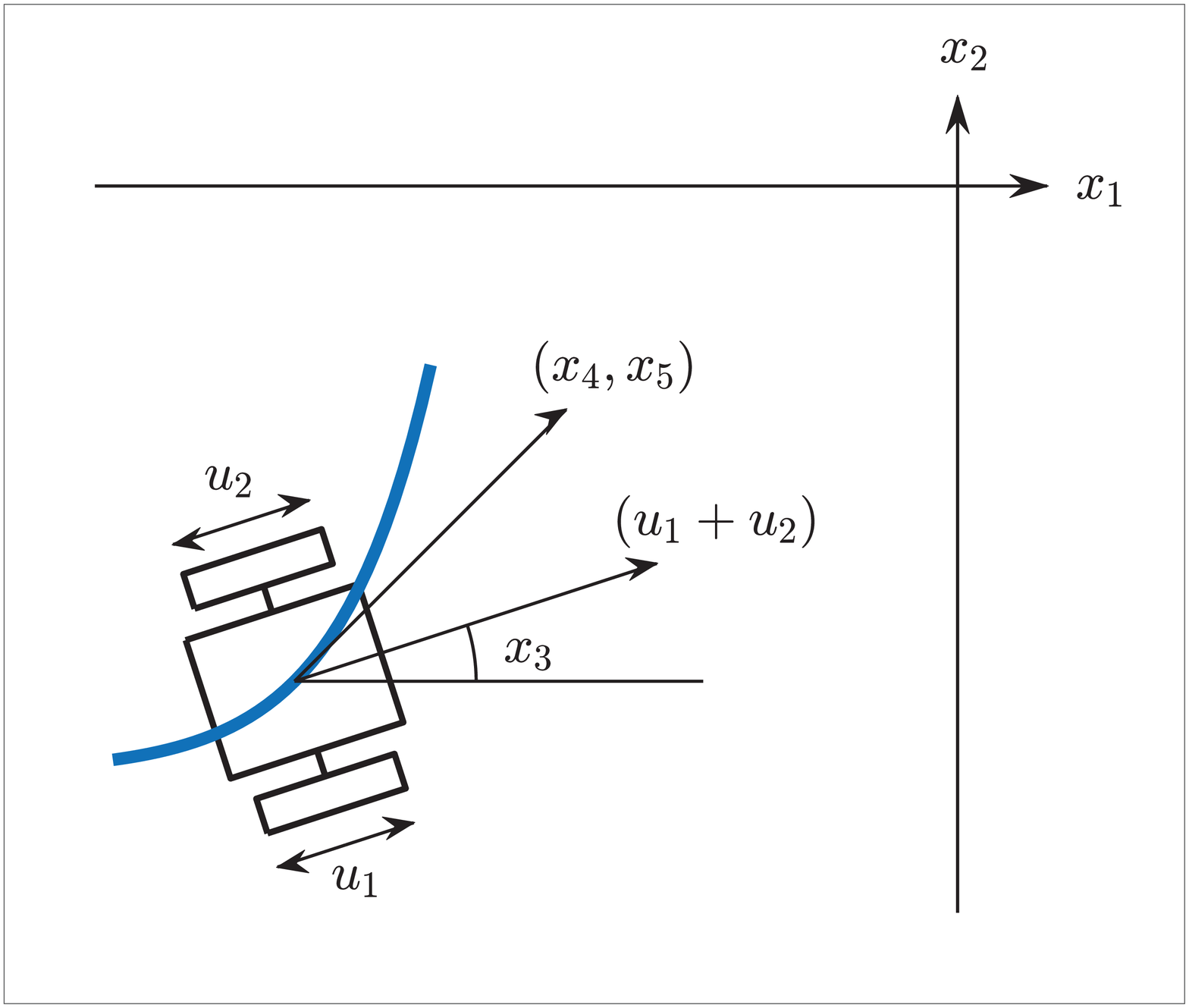} \\[4mm]
  \caption{\small The free-flying robot.}
 \label{fig:ffRobot}
\end{subfigure}
\hspace*{10mm}
\begin{subfigure}{.45\linewidth}
\includegraphics[width=1\linewidth]{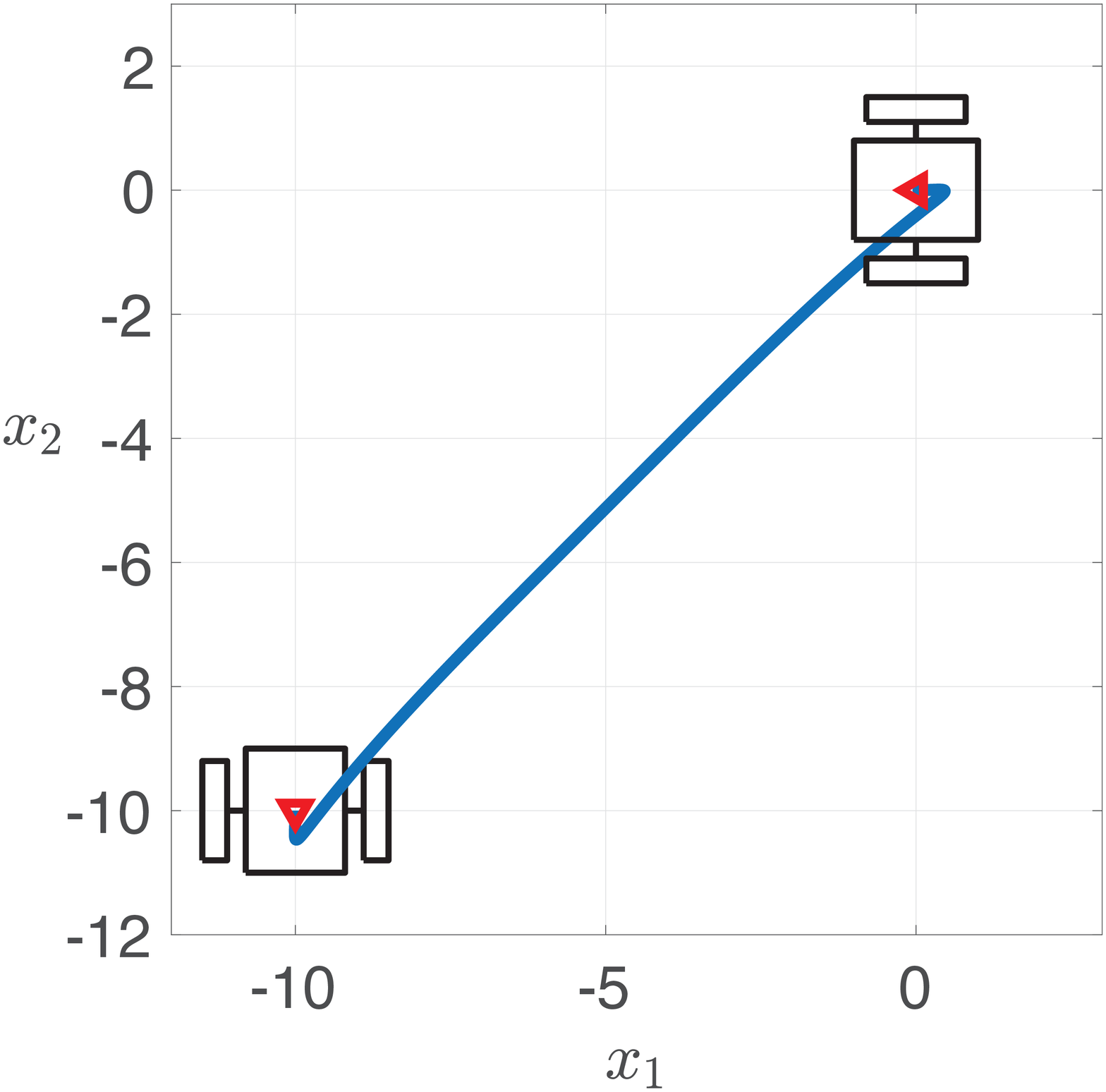}
  \caption{\small Optimal trajectory of the free-flying robot.}
  \label{fig:trjectory}
  \end{subfigure}
\caption{\sf Configuration and solution trajectory of the free flying robot.}  
\label{fig:robots}
\end{figure}

The model, as studied in \cite{BY2013,BYir,VosMau2006}, is as follows.
\[
(P3)\,
  \left\{
    \begin{array}{rl}
      \min & \ds\int_0^{12} \left(u_1^2(t)+u_2^2(t)\right)dt  \\[4mm]
      {\rm s.t. } &\ \dot x_1(t)=x_4(t)\,, \\[1mm]
      &\ \dot x_2(t)=x_5(t)\,, \\[1mm]
      &\ \dot x_3(t)=x_6(t)\,, \\[1mm]
      &\ \dot x_4(t)=\left(u_1(t)+u_2(t)\right)\cos x_3(t)\,, \\[1mm]
      &\ \dot x_5(t)=\left(u_1(t)+u_2(t)\right)\sin x_3(t)\,, \\[1mm]
      &\ \dot x_6(t)= 0.2\left(u_1(t)-u_2(t)\right), \\[1mm]
      &\ x(0)=(-10,-10,\pi/2,0,0,0),\ \ x(12)=(0,0,0,0,0,0)\,, \\[1mm]
      &\ |u_1(t)|\le 0.8\,, \quad |u_2(t)|\le 0.4\,.
    \end{array}
  \right.
\]

\newpage
\subsection{Formulation of the Free-flying robot problem}

To be able to apply the PDP algorithm for solving the FFR problem, we need to formulate $(P3)$ in the format \eqref{eq:generalPr} of Theorem \ref{th:example_verify_H}. With the notation of that theorem, take $U:={\mathcal L}^2([0,1];\R)\times {\mathcal L}^2([0,1];\R)$ and $H:=\R^6$. The box constraints on $u$ can be expressed using the set $K_2:=\{u\in {\mathcal L}^2([0,1];\R)\times {\mathcal L}^2([0,1];\R)\::\: |u_1(t)|\le 0.8,\, |u_2(t)|\le 0.4, \,\forall\, t\in [0,1]\}$. As in Section \ref{sec:DoubleIntegrator}, our first step is to show that the ODE system appearing in the constraints of $(P3)$ can be equivalently reformulated as $h(u)=0$ for a suitable function $h:U\to \R^6$. This fact is established in the next lemma. The idea, which is elementary albeit laborious, is to integrate the ODE system.

\begin{lemma}[ODEs as equality constraints]\label{lem:h for P3}
 There exists a function of $h:{\mathcal L}^2([0,1];\R)\times {\mathcal L}^2([0,1];\R)\to \R^6$ such that the ODE system in $(P3)$ can be written as $h(u)=0$.   
\end{lemma}
\begin{proof}
    See the proof of  Lemma~\ref{lem:h for P3} in \hyperlink{prf:h for P3}{Appendix}.
\end{proof}

Recall our notation $K_2:=\{(u_1,u_2)\in {\mathcal L}^2([0,1];\R)\times {\mathcal L}^2([0,1];\R)\::\: |u_1(t)|\le 0.8,\, |u_2(t)|\le 0.4, \,\forall\, t\in [0,1]\}$. To verify Assumption (b) in Theorem \ref{th:example_verify_H}, we need to show that $K_2\cap h^{-1}(z)$ is w-closed, for $h$ as in Lemma \ref{lem:h for P3}. We establish this in the next lemma.

\begin{lemma}
    \label{lem:cond b for P3}
Consider the ODE system and the corresponding boundary conditions given for Problem $(P3)$ and let $h:{\mathcal L}^2([0,1];\R)\times {\mathcal L}^2([0,1];\R)\to \R^6$ be as in  Lemma \ref{lem:h for P3}. Then the set
$K_2\cap h^{-1}(z)$ is w-compact and hence w-closed. 
\end{lemma}
\begin{proof}
 See the proof of  Lemma~\ref{lem:cond b for P3} in \hyperlink{prf:cond b for P3}{Appendix}.
\end{proof}

\begin{corollary}[Problem $(P3)$ verifies (H0)--(H2)]
    \label{lem:DP for (P3)}
Let $h$ be as in Lemma \ref{lem:h for P3}. Consider for Problem $(P3)$ the dualizing parametrization $f:{\mathcal L}^2([0,1];\R)\times {\mathcal L}^2([0,1];\R) \times \R^6\to \R_{\infty}$ defined by
\[
f(u,z):=\varphi(u)+\delta_{z}(h(u))+\delta_K(u)= \varphi(u)+\delta_{h^{-1}(z)\cap K}(u).
\]  
Then assumptions {\rm\hyperlink{H0}{(H0)}--\hyperlink{H2}{(H2)}} hold for Problem $(P3)$.
\end{corollary}
\begin{proof} 
The verification of assumptions (a) and (c) of Theorem \ref{th:example_verify_H} for $(P3)$ is identical to the one in Theorem \ref{th:verify_p1}. Assumption (b) follows from Lemma \ref{lem:cond b for P3} and the fact that $\Gamma(0)\supset S(P)\neq \emptyset$ by \cite[Section 6.2]{BYir}. Indeed, the latter paper shows that there is a unique solution $u$ of $(P3)$. By Theorem \ref{th:example_verify_H}, we conclude that $(P3)$ verifies {\rm\hyperlink{H0}{(H0)}--\hyperlink{H2}{(H2)}}.
\end{proof}

\subsubsection{Optimality Conditions for Problem $(P3)$}
{
We consider the optimality conditions for Problem $(P3)$ as computed in Section~\ref{sec:optimalityCondn}. 
The Hamiltonian function $H:\R^6\times\R^2\times\R^6\to\R$ for Problem $(P3)$ is
\begin{equation*}\label{HamiltonianRobot}
    \begin{array}{lll}
          H(x,u,\lambda) & = &  
          u_1^2+u_2^2+\lambda_1x_4+\lambda_2x_5+\lambda_3x_6
    + \lambda_4\left(u_1+u_2\right)\,\cos x_3 \\[1mm]
          & & +\ \lambda_5\left(u_1+u_2\right)\sin x_3
    +0.2\lambda_6\left(u_1-u_2\right),
    \end{array}
\end{equation*}
where the state variable vector $x(t) = (x_1(t),\ldots,x_6(t)) \in\R^6$, the control variable vector ${u}(t) =(u_1(t),u_2(t))\in\R^2$.  The adjoint variable vector $\lambda(t)=(\lambda_1(t),\ldots,\lambda_6(t))\in \R^6$ satisfies, by Equation \eqref{eq:adjoint},
\begin{equation*}\label{AdjointFFR}
\begin{array}{l}
     \lambda_1(t)=c_1, \quad    \lambda_2(t)=c_2, \quad    \lambda_4(t)=-c_1t+c_4,
     \quad    \lambda_5(t)=-c_2t+c_5,  \quad  
     \dot \lambda_6(t) = \lambda_3(t),  \quad \mbox{and}\\[1mm]
     \dot \lambda_3(t) = \lambda_4(t)\left(u_1(t)+u_2(t)\right)\,\sin x_3(t)
          -\lambda_5(t)\left(u_1(t)+u_2(t)\right)\cos x_3(t),    
\end{array}
\end{equation*}
for all $t\in[0,1]$, where $c_1$, $c_2$, $c_4$ and $c_5$ are real constants. By Equation~\eqref{eq:controlOptCdn}, we obtain the optimal control variables as follows
\begin{equation}\label{eq:u1controlFFR}
u_1(t)=\,
  \left\{
    \begin{array}{lcl}
    -\psi_1(t) & , & \mbox{if}\,\, -0.8\leq\psi_1(t)\leq 0.8,\\[2mm]
    \ \ \,0.8 & , & \mbox{if}\,\, \psi_1(t) \leq -0.8,\\[2mm]
    -0.8 & , & \mbox{if}\,\, \psi_1(t) \geq 0.8,\\
    \end{array}
  \right.
\end{equation}
where the {\em switching function for $u_1$} is given by $\psi_1(t):=\dfrac{1}{2}\left(\lambda_5(t)\cos x_3(t)+\lambda_4(t)\sin x_3(t)+0.2\lambda_6(t)\right)$;
\begin{equation}\label{eq:u2controlFFR}
u_2(t)=\,
  \left\{
    \begin{array}{lcl}
    -\psi_2(t) & , & \mbox{if}\,\, -0.4\leq \psi_2(t)\leq 0.4,\\[2mm]
    \ \ \,0.4 & , & \mbox{if}\,\, \psi_2(t) \leq -0.4,\\[2mm]
    -0.4 & , & \mbox{if}\,\, \psi_2(t) \geq 0.4,\\
    \end{array}
  \right.
\end{equation}
where the {\em switching function for $u_2$} is given by $\psi_2(t):=\dfrac{1}{2}\left(\lambda_5(t)\cos x_3(t)+\lambda_4(t)\sin x_3(t)-0.2\lambda_6(t)\right)$.
}

\subsubsection{Numerical solution of Problem $(P3)$}
We discretize and solve Problem $(P3)$ numerically as described in Section~\ref{sec:discritization}. We use the PDP algorithm with the two choices of step-sizes we have proposed, and we use Ipopt on its own to solve Problem $(P3)$, running each of the methods 1000 times in order to get reliable statistics. We take 
random initial guesses generated uniformly in given intervals, such that
\[
x_{pi} \in [-0.4, 0.4] \quad \mbox{and}\quad u_{rj}\in[-0.4, 0.4], 
\]
for $p = 1,\dots , 6$, $r = 1, 2$, $i = 0, \dots , N$, $j = 0, \dots , N-1$. The results for $x(\cdot)$ and $\lambda(\cdot)$ are shown in Figure~\ref{fig:stateCostateFFR}, while those for $u_i(\cdot)$ and their switching functions $\psi_i(\cdot)$, $i = 1,2$, are displayed in Figure~\ref{fig:upsiFFR}.  The graphs in Figure~\ref{fig:upsiFFR} play the role of a certificate verifying the optimality conditions given in \eqref{eq:u1controlFFR} and \eqref{eq:u2controlFFR}.  Recall that we had already included the trajectory of the free-flying robot earlier in the $x_1x_2$-plane in Figure~\ref{fig:trjectory}.

 \begin{figure}[t]
  \begin{subfigure}{.48\textwidth}
\centering
  \includegraphics[width=\textwidth]{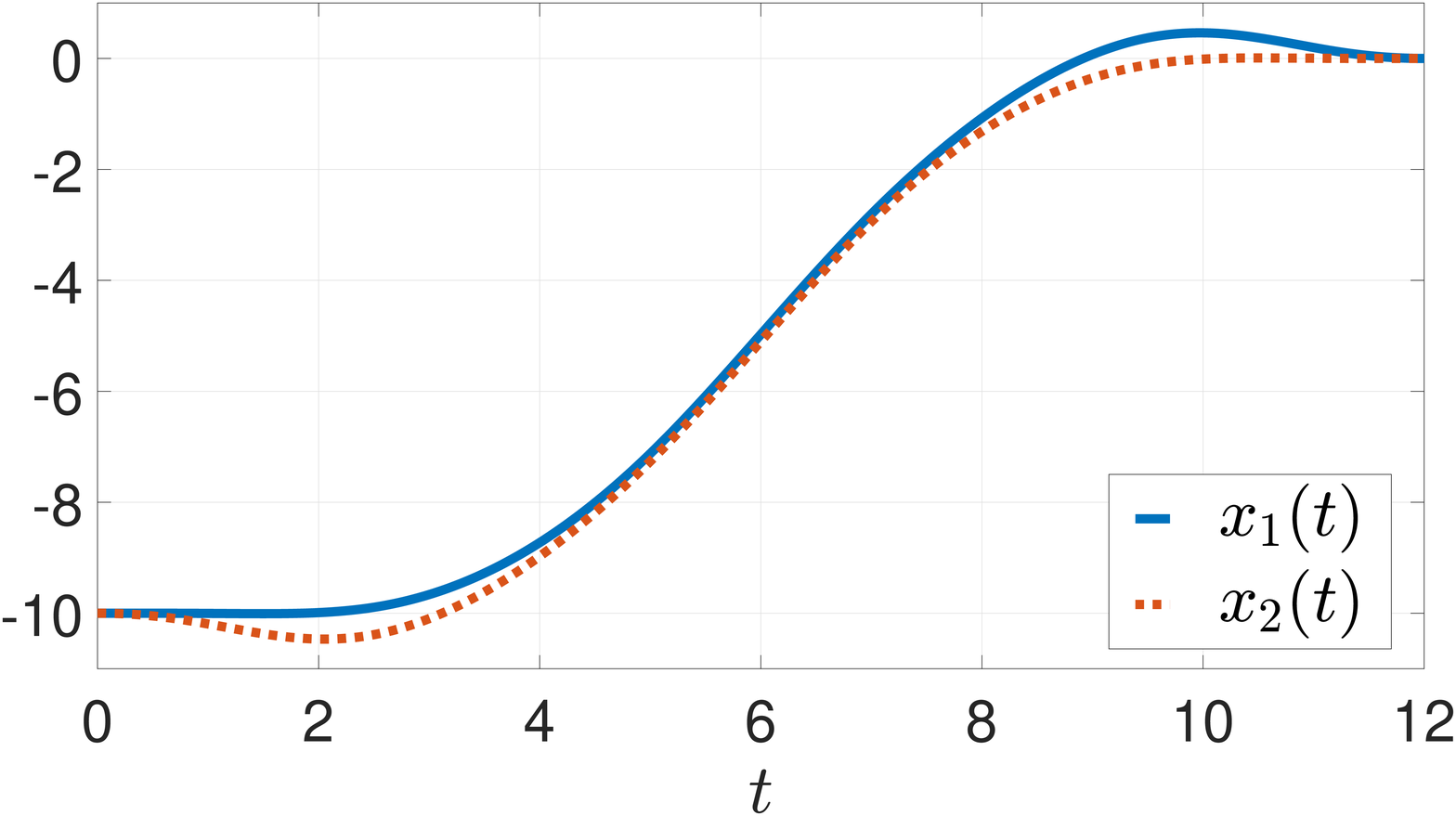}
  \label{fig:xFigUnconstr}
\end{subfigure}
\hfill
\begin{subfigure}{.48\textwidth}
 \centering
  \includegraphics[width=\textwidth]{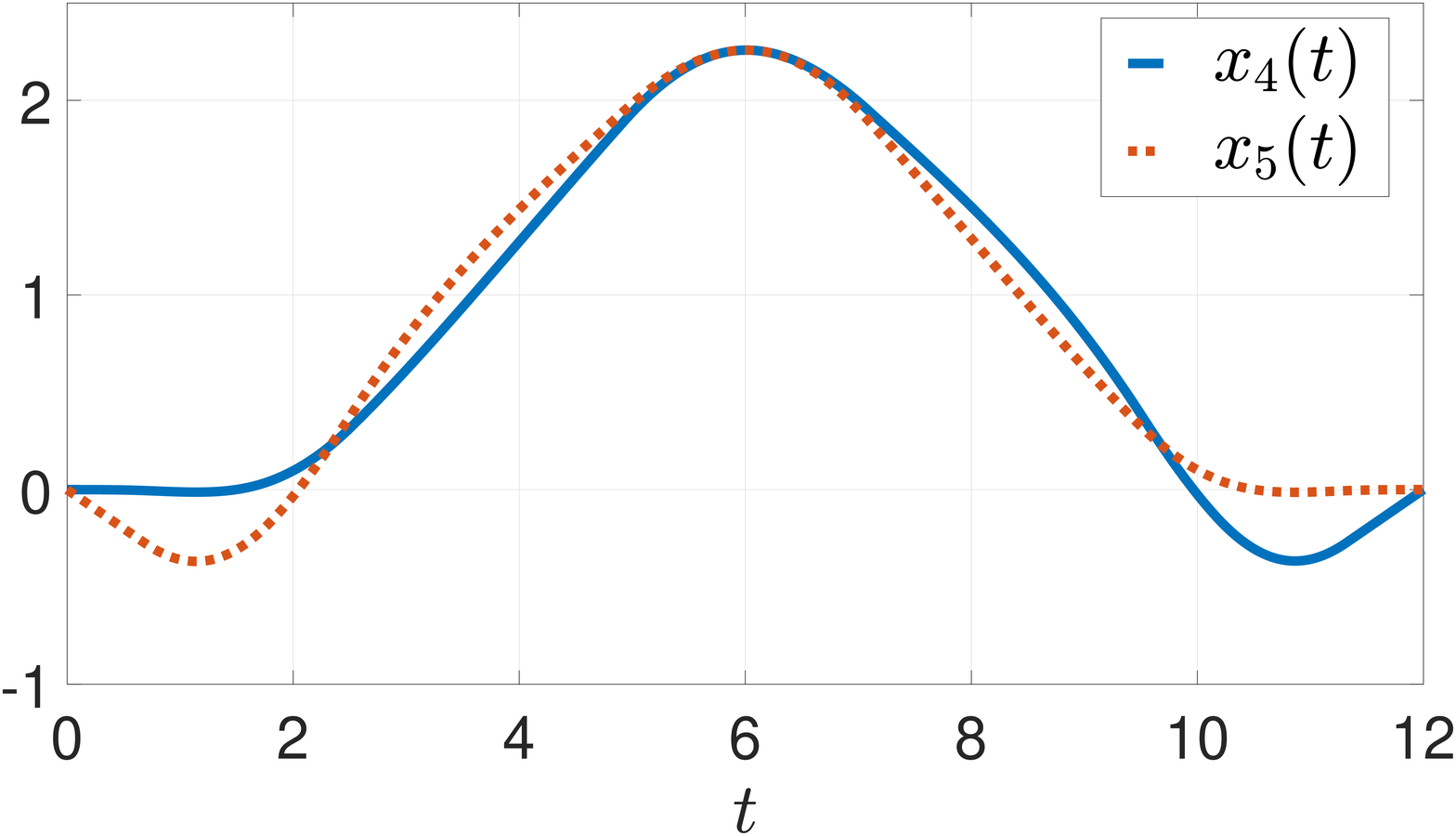}
  \label{fig:uFigUnconstr}
  \end{subfigure}
  \label{fig:uFigUnconstrr}\\
    \begin{subfigure}{.48\textwidth}
\centering
  \includegraphics[width=\textwidth]{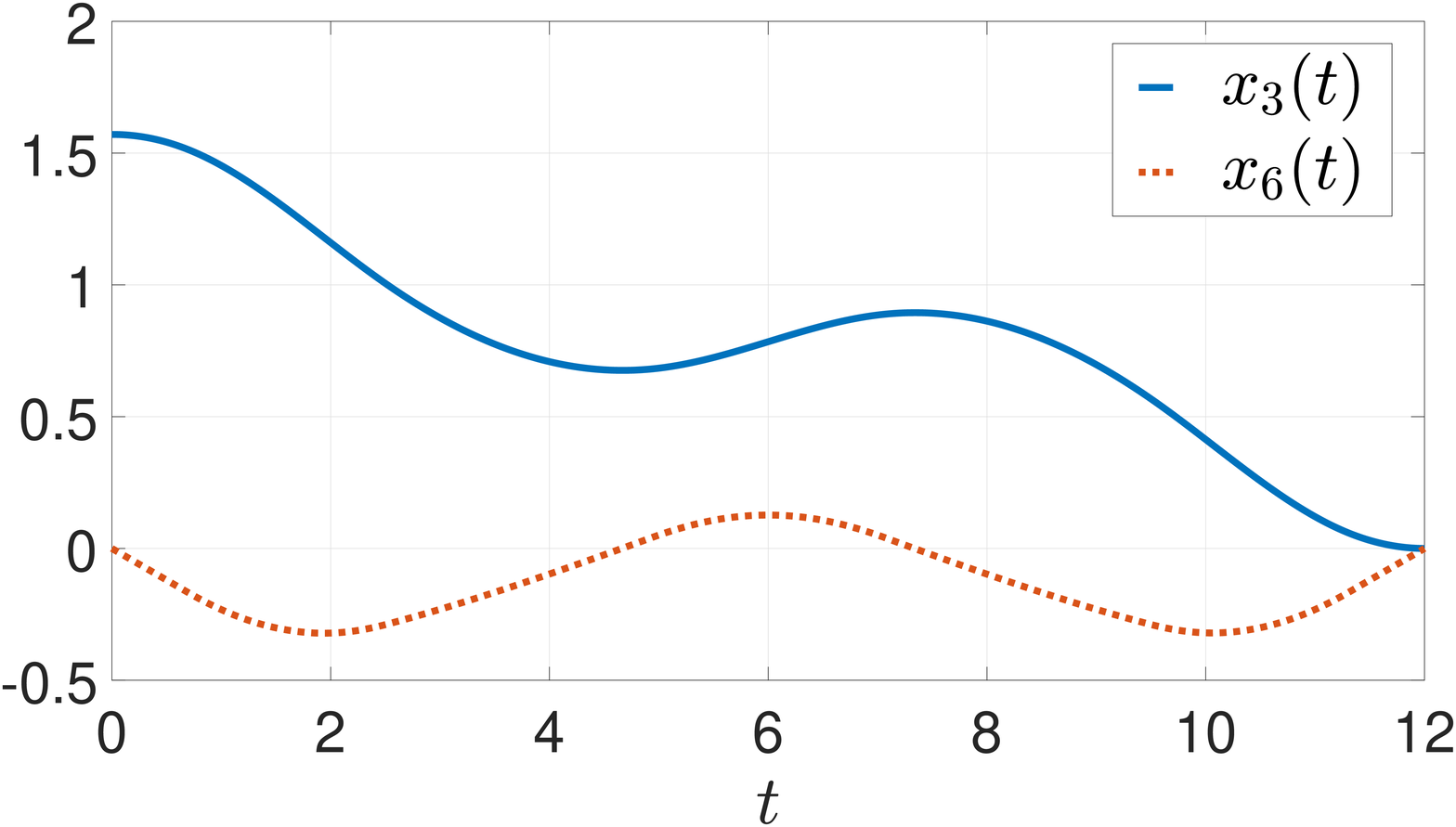}
  \label{fig:stateFigUnconstr}
\end{subfigure}
\hfill
\begin{subfigure}{.48\textwidth}
 \centering
  \includegraphics[width=\textwidth]{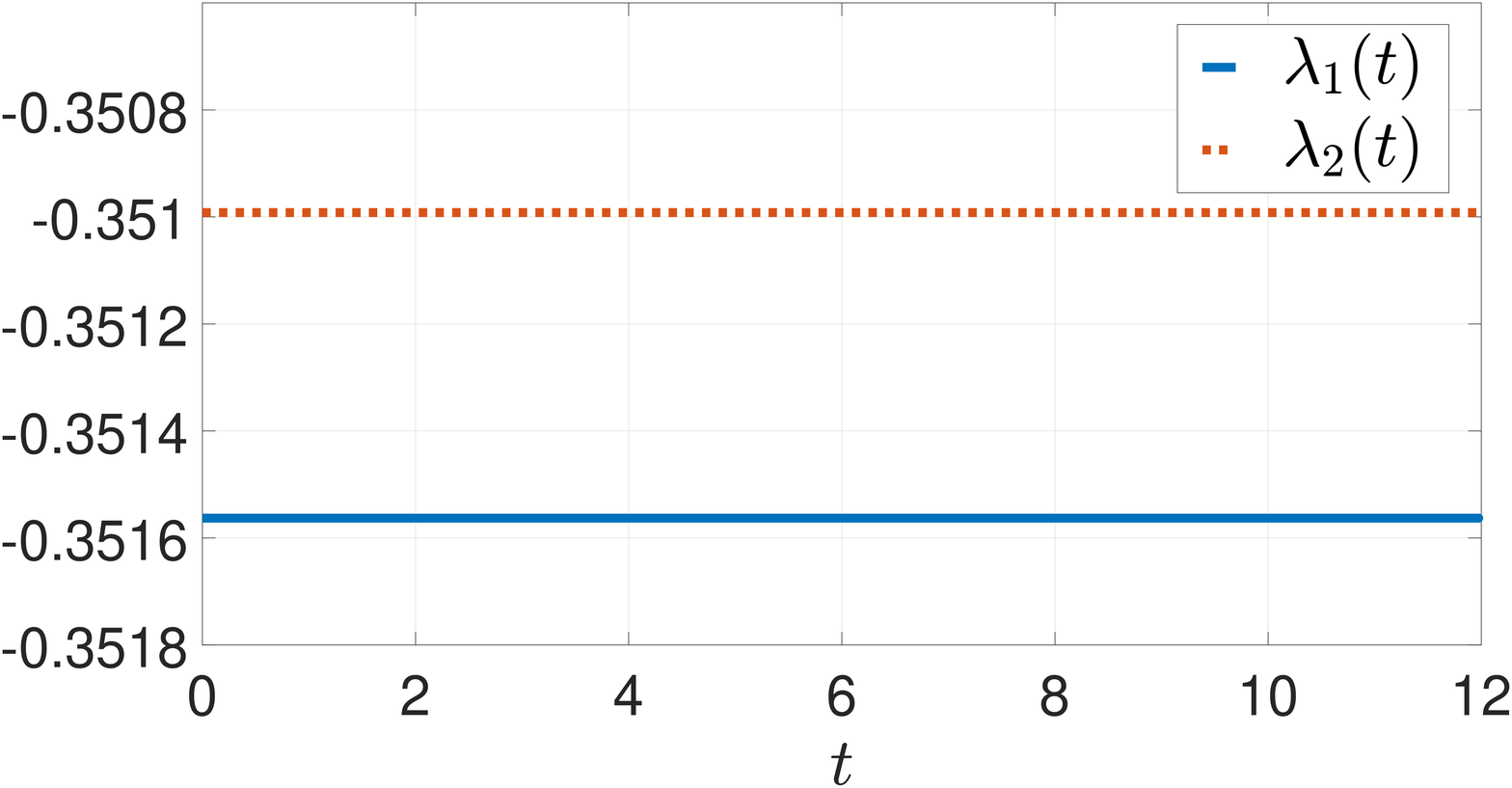}
  \label{fig:uFigUnconstr2}
  \end{subfigure}\\
      \begin{subfigure}{.48\textwidth}
\centering
  \includegraphics[width=\textwidth]{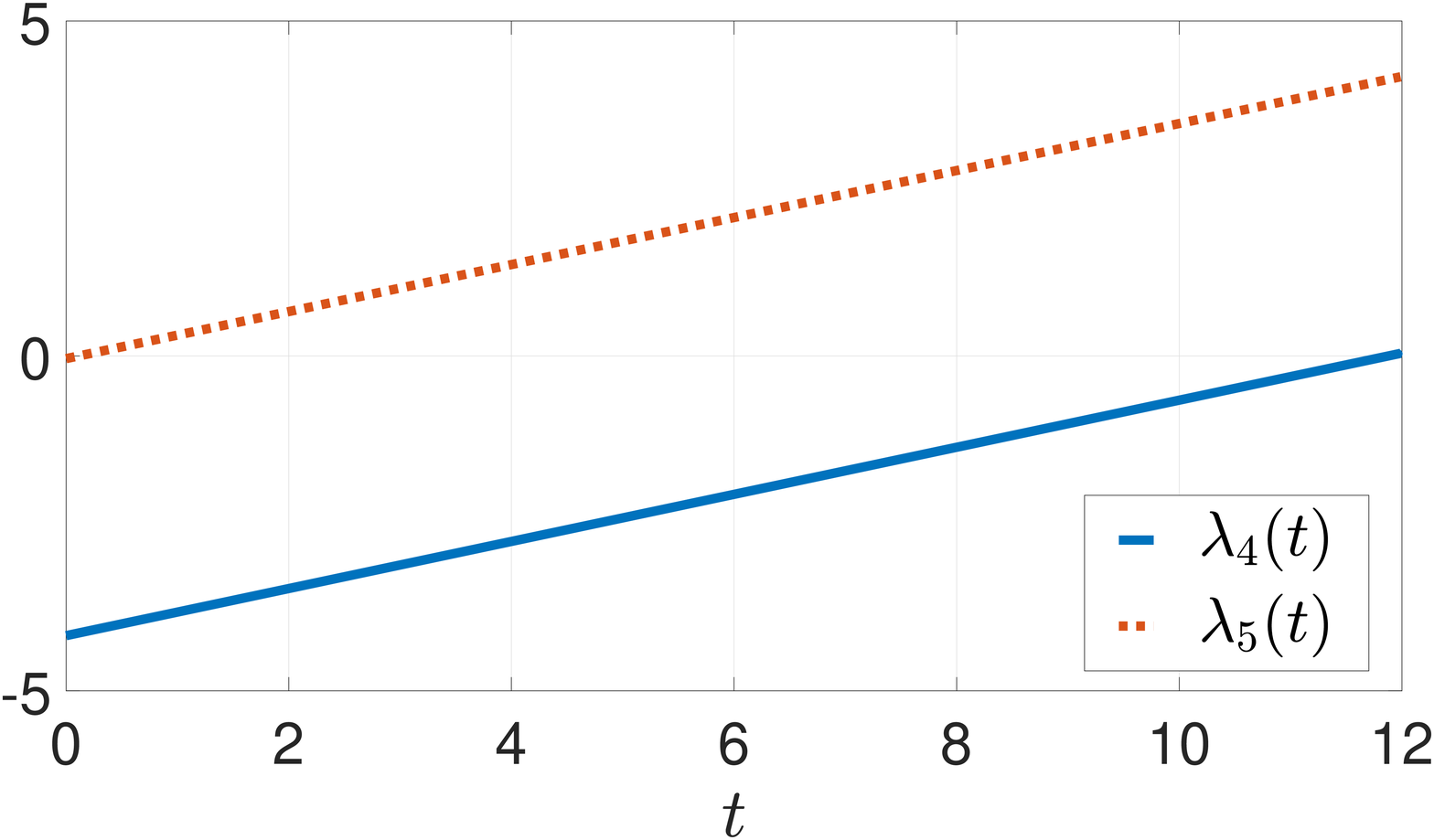}
  \label{fig:xFigUnconstr2}
\end{subfigure}
\hfill
\begin{subfigure}{.48\textwidth}
 \centering
  \includegraphics[width=\textwidth]{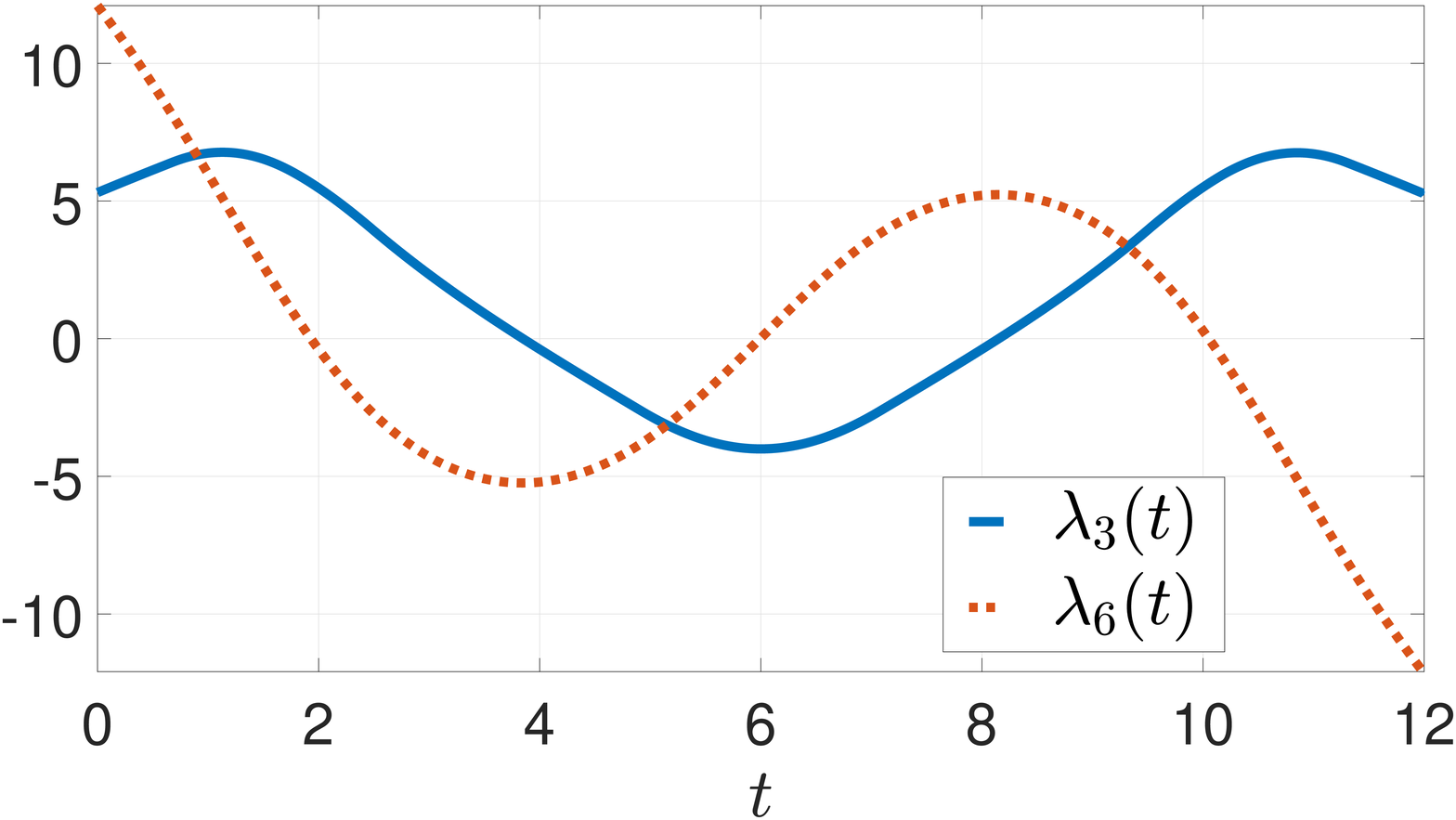}
  \label{fig:uFigUnconstr3}
  \end{subfigure}\
  \caption{\sf Problem $(P3)$---Optimal state and adjoint variables.}
  \label{fig:stateCostateFFR}
\end{figure}

In the PDP algorithm with the step-size in \eqref{sk_DSG1}, we have taken
\[
s_k=(1+\alpha_k)\left[\dfrac{1}{2}\min(\eta, \|h(u_k)\|_2) + \dfrac{1}{2}\max(\beta, \|h(u_k)\|_1 + \|h(u_k)\|_2)\right],
\]
where $\alpha_k=0.4$ for all $k$, $\eta= 0.1$, and $\beta=1$, and 
$h(u_{c_k})$ is the constraint function $h$ at the current iterate $u_{c_k}$, for $h$ as given in \eqref{app:h} in the proof of Lemma \ref{lem:h for P3}. In the PDP algorithm with the step-size in \eqref{sk_DSG2}, we have taken $s_k=(1+\alpha_k)\gamma_k$, where
\[
\gamma_k\in \left[\dfrac{\theta_k}{\|h(u_k)\|_1},\dfrac{\beta}{\|h(u_k)\|_1}\right]\,,
\]
and $\alpha_k=1$, $\theta_k=1$ for all $k$, and $\beta=2$.
The resulting dual function value iterates of the PDP algorithm, superimposed with the numerically computed graph of the dual function for Problem~$(P3)$, are displayed in Figure~\ref{fig:iterationFFR}. 

In Figure~\ref{fig:uApproxEg3}, we illustrate the iterations of $u(t)$ using the same step-size of type 2 as in Figure~\ref{fig:freeFlyRobot2}. The performance of each approach, with randomly generated initial guesses, is presented in Table~\ref{table:FFRobot}. Note that the CPU time in column 5 for Ipopt alone corresponds to the average CPU time for all runs, successful or unsuccessful.  The CPU time for Ipopt alone in column 8 (last column) corresponds to the average CPU time for successful runs only.

When compared with the case of using Ipopt on its own, the PDP method achieves $100\%$ success rate at all levels of discretization shown in Table \ref{table:FFRobot}. Moreover, PDP has a better performance in terms of CPU time when the number of discretization points is greater than $2000$ and Ipopt is successful. When compared with the inexact restoration algorithm proposed in \cite{BYir}, the PDP method has a similar success rate and a better performance in terms of the CPU time, although the codes of either approach were run on different computers.

\begin{figure}[t]
\centering
  \begin{subfigure}{.49\linewidth}
\includegraphics[width=\linewidth]{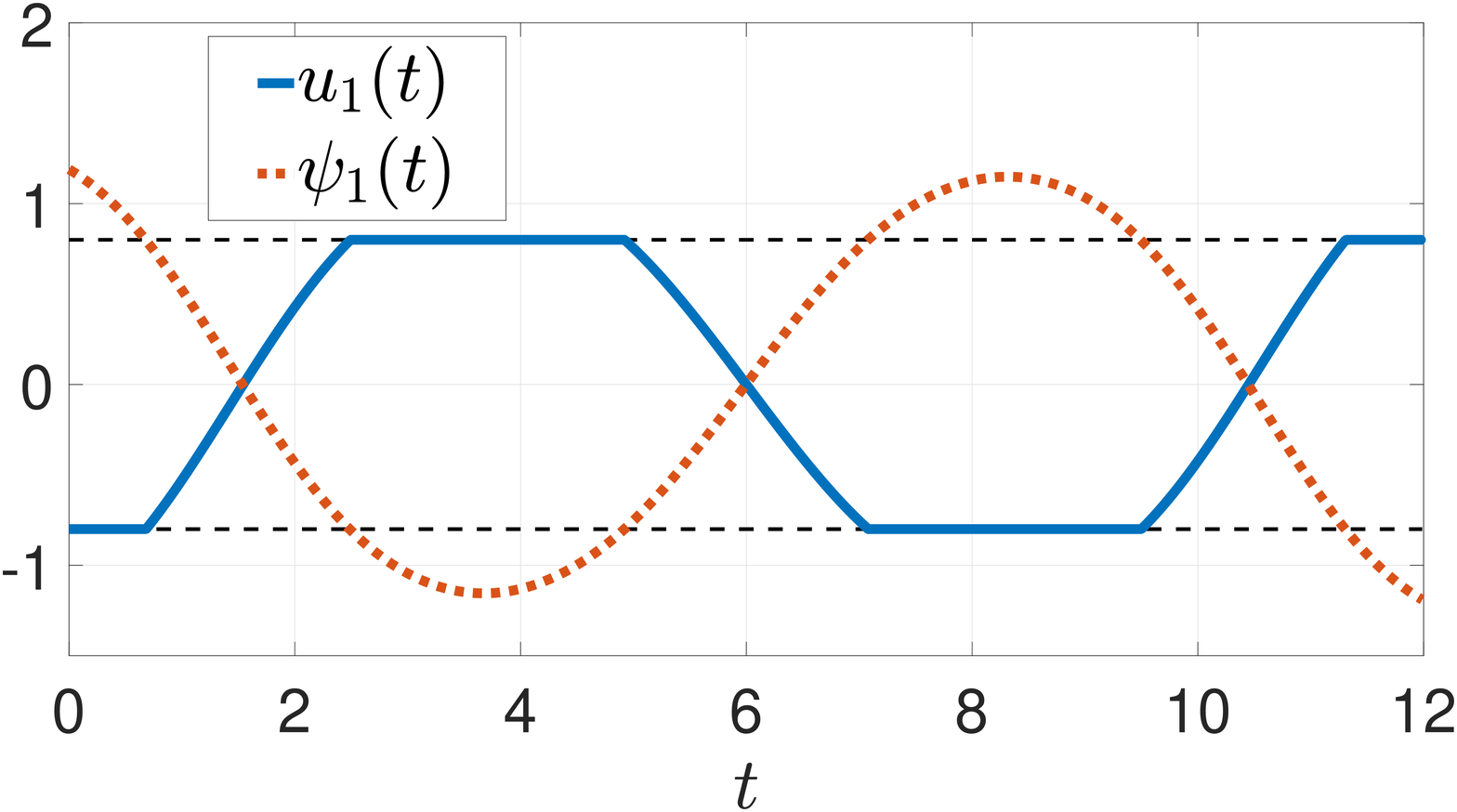}
  \caption{\small Optimal control $u_1$ and its switching function $\psi_1$.}
 \label{fig:uFigUnconstr4}
\end{subfigure}
\centering
\begin{subfigure}{.49\linewidth}
\includegraphics[width=\linewidth]{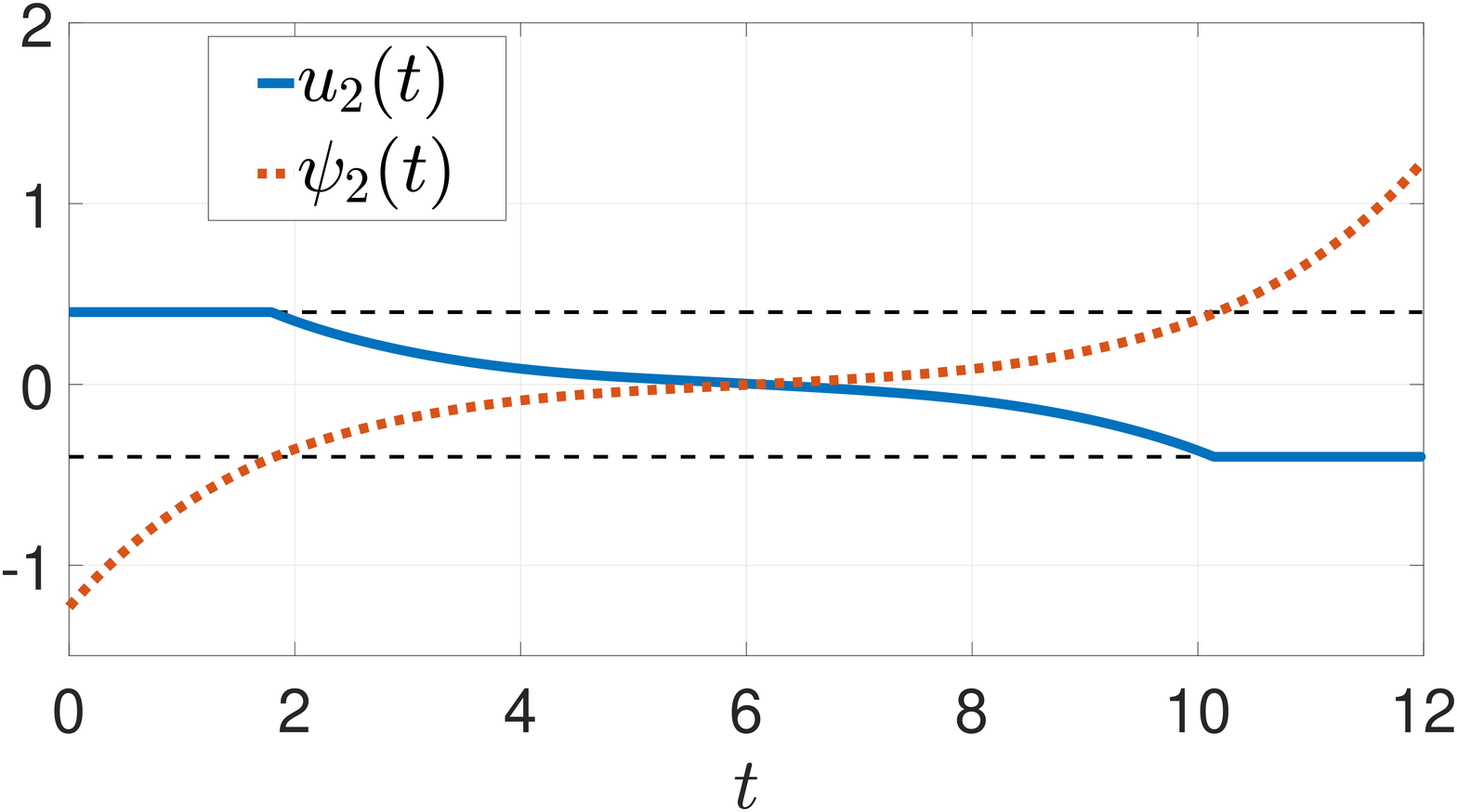}
 \caption{\small Optimal control $u_2$ and its switching function $\psi_2$.}
  \label{fig:u2psi2}
  \end{subfigure}
  \caption{\sf Problem $(P3)$---The optimal control variables and their switching functions.}
  \label{fig:upsiFFR}
\end{figure}

\begin{figure}[ht]
  \centering
  \begin{subfigure}{.49\linewidth}
\includegraphics[width=\linewidth]{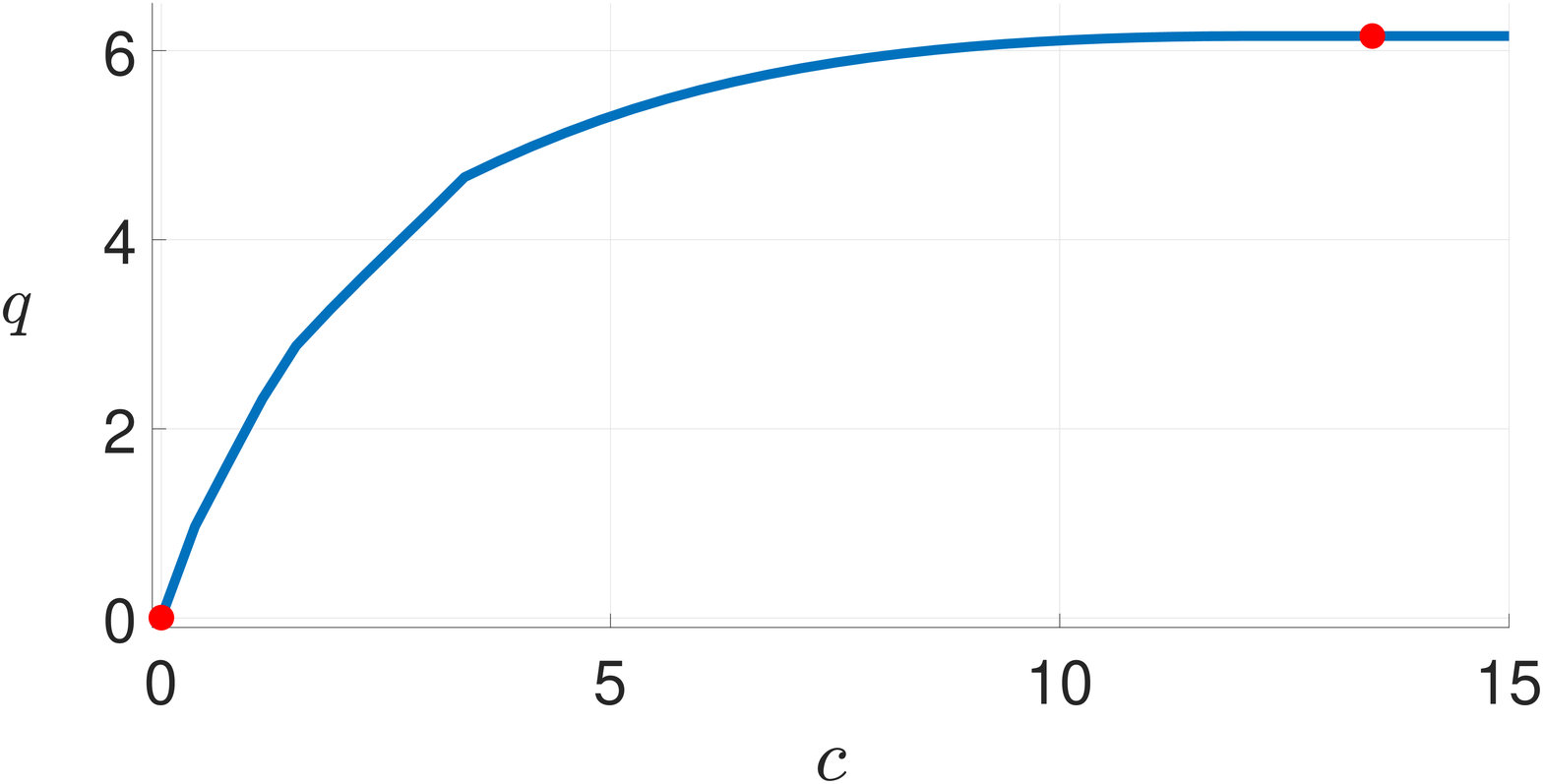}
  \caption{\small Algorithm~PDP with the step-size in \eqref{sk_DSG1}.}
  \label{fig:freeFlyRobot1}
\end{subfigure}
\begin{subfigure}{.49\linewidth}
\includegraphics[width=\linewidth]{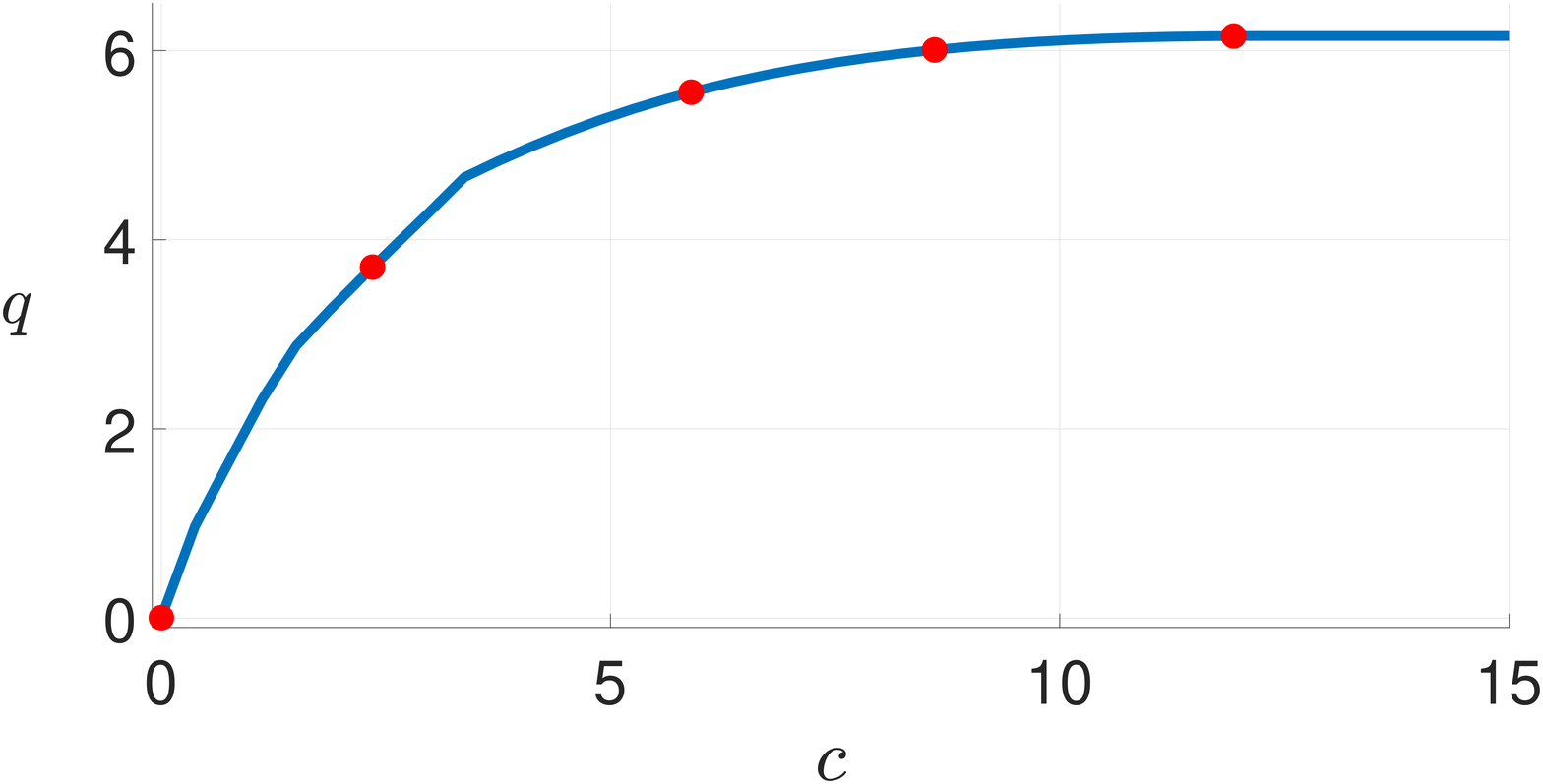}
  \caption{\small Algorithm~PDP with the step-size in \eqref{sk_DSG2}.}
  \label{fig:freeFlyRobot2}
  \end{subfigure}
  \caption{\sf Problem $(P3)$---The dual function updates (shown by red dots on the blue curve representing the graph of the dual function) in each iteration of Algorithm~PDP using step-sizes of type 1 and 2.}
  \label{fig:iterationFFR}
\end{figure}
 We observe that when using Ipopt alone to solve the problem, the success rate decreases as $N$ increases. This is because in practice, the computation of a solution becomes harder as the number of optimization variables increases. When we use the PDP algorithm, the ODE system is not a hard constraint anymore, but instead it is reflected in the objective function of the minimization step. The fact that the ODE system is no longer a hard constraint seems to have a beneficial effect in terms of CPU time. More experimentation, with different types of problems, however, is needed to determine precisely whether or not this is the reason for the better performance in terms of CPU of PDP. We use different number of discretization $N$ to compute $u$ for Problem $(P3)$ by our PDP algorithm with the same step-size of type 2 as in Figure~\ref{fig:freeFlyRobot2}. 
 
 We plot the solution of $u$ with $N=39, 100$ and $\infty$ (the case of $\infty$ is represented by the solution obtained by $N=1000$) in Figure~\ref{fig:uVSnEg3}. The second coordinate $u_2$  seems to be more sensitive to the number of discretization points than the first coordinate $u_1$. Since the curves obtained by the PDP algorithm and those obtained by Ipopt alone are indistinguishable from each other, we only show in Figure~\ref{fig:uVSnEg3} those generated by the PDP algorithm. In the implementation of the PDP algorithm, $N=38$ does not seem to be large enough to observe the true solution pattern with the correct number of junction points in time and the more-or-less correct locations of the junctions, while Ipopt alone seems to yield solutions with correct pattern for $N\le 38$. However, as discussed for Table~\ref{table:FFRobot}, Ipopt alone fails to solve the problem over half of the time and uses more average CPU (when successful) than the PDP algorithm when $N\ge 1000$.

\begin{table}[t!]
\centering
{\small
\begin{tabular}{r | c  c  c | r  c  c | c}
 & \multicolumn{3}{c|}{success rate [\%] }& \multicolumn{3}{c|}{Ave.~CPU time [sec]} 
 & Ave.~CPU time \\
\cline{2-4} \cline{5-7}
 N \ \ \  & Ipopt & \multicolumn{2}{c|}{PDP} & Ipopt & \multicolumn{2}{c|}{PDP} & for Ipopt alone\\
 \cline{3-4} \cline{6-7}
 & alone & $s_k$ \eqref{sk_DSG1} & $s_k$ \eqref{sk_DSG2} & alone & $s_k$ \eqref{sk_DSG1} & $s_k$ \eqref{sk_DSG2} & (when successful) [sec]\\ \hline
 $100$ & $80$ & 100 & 100 & $0.2$ & $0.2$ & $0.6$ & $0.1$\\ 
 [2mm]
 $500$ & $70$ & 100 & 100 & $2.4$ & $0.9$ &  $3.1$ & $0.9$ \\ 
 [2mm]
 $1000$ & $50$ & 100 & 100 & $6.0$ & $1.6$ & $5.8$ & $2.2$ \\ 
 [2mm]
 $2000$ & $30$ & 100 & 100 & $22.7$ & $3.8$ & $12.2$ &$5.2$\\ 
 [2mm]
 $5000$ & $20$ & 100 & 100 & $108.7$ & $9.8$ & $30.8$ & $17.1$\\ 
 [2mm]
 $10000$ & $10$ & 100 & 100 & $517.6$ & $16.3$ & $62.3$ & $46.4$\\ 
 \hline
 \end{tabular}
 }
\caption{\small Problem $(P3)$---Success rates $[\%]$ correct to one significant figure for the PDP algorithm and Ipopt alone for the free-flying robot.}
\label{table:FFRobot}
\end{table}

\begin{figure}[t!]
  \centering
  \begin{subfigure}{.49\linewidth}
\includegraphics[width=\linewidth]{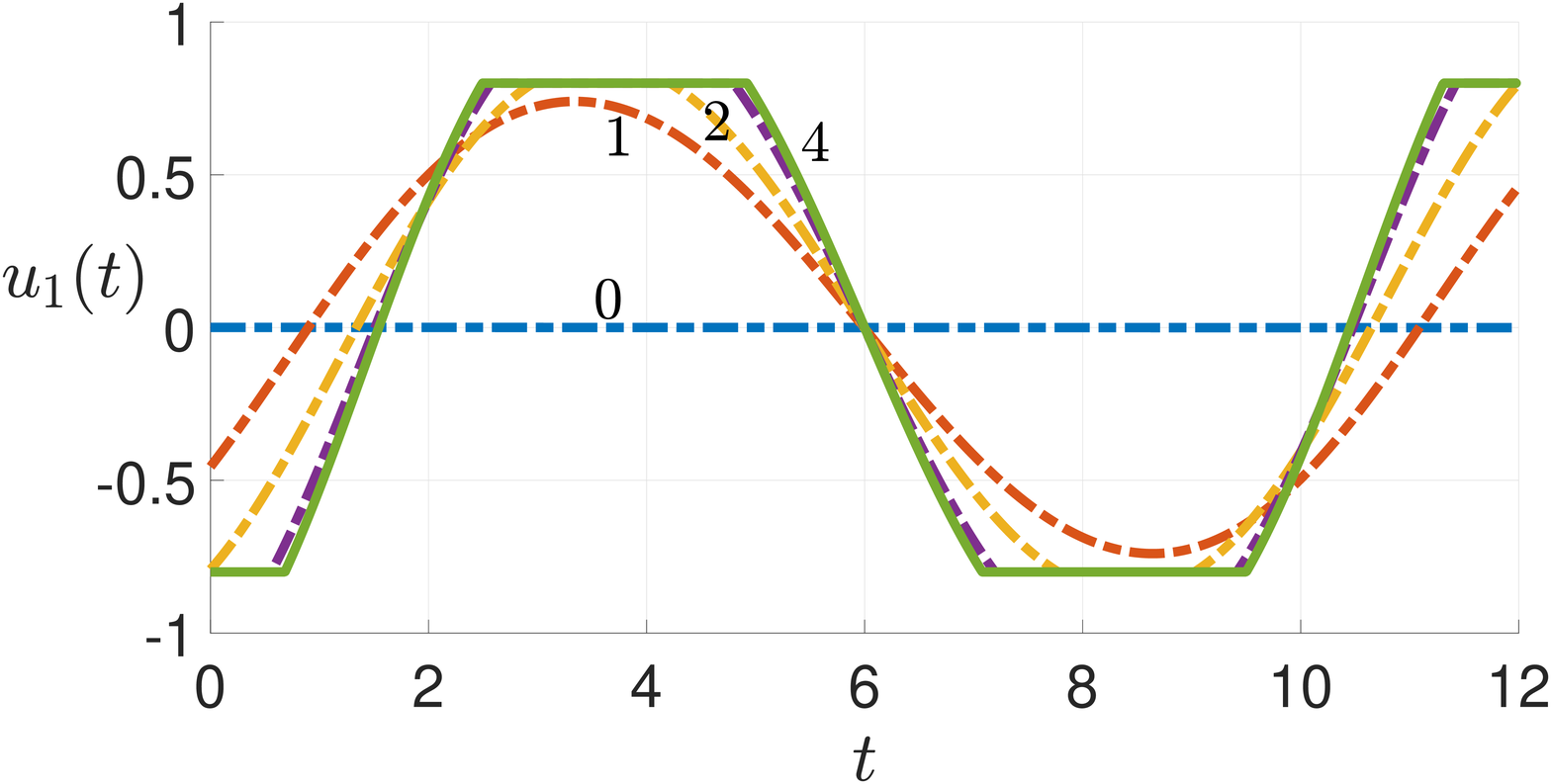}
  \caption{\small Iterations of $u_1$ with $s_k$ in \eqref{sk_DSG2}.}
  \label{fig:u1ApproxEg3}
\end{subfigure}
\begin{subfigure}{.49\linewidth}
\includegraphics[width=\linewidth]{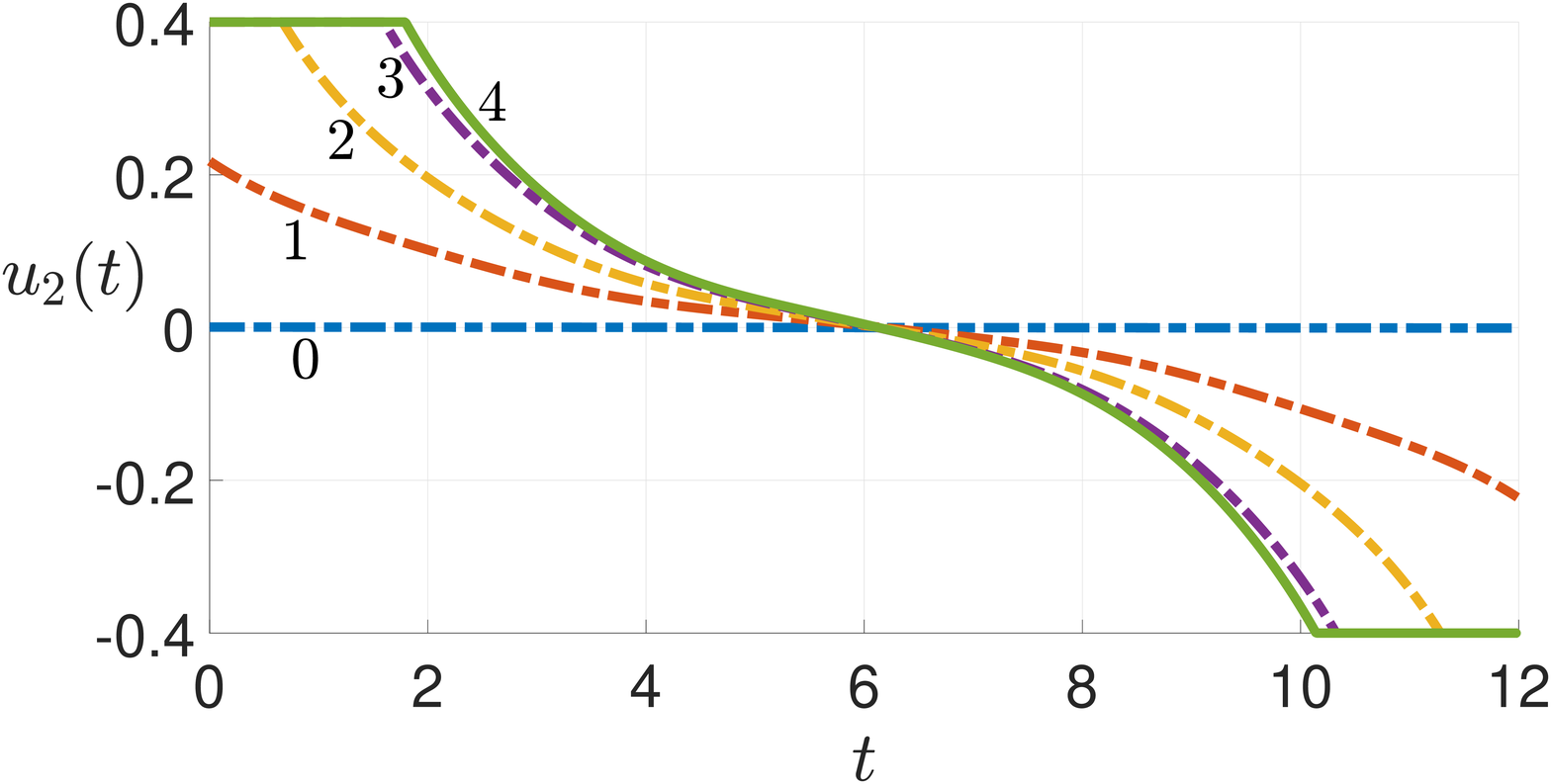}
  \caption{\small Iterations of $u_2$ with $s_k$ in \eqref{sk_DSG2}.}
  \label{fig:u2ApproxEg3}
  \end{subfigure}
  \caption{\sf Problem $(P3)$---The iterations of $u$ labelled 0--4, standing for $u_{i,j}$, $i=1,2$ and $j=0,1,2,3,4$, obtained by algorithm PDP-2. In (a), $u_{1,3}$ is not labelled for clarity in the appearance.}
  \label{fig:uApproxEg3}
\end{figure}

\begin{figure}[t!]
  \centering
  \begin{subfigure}{.49\linewidth}
\includegraphics[width=\linewidth]{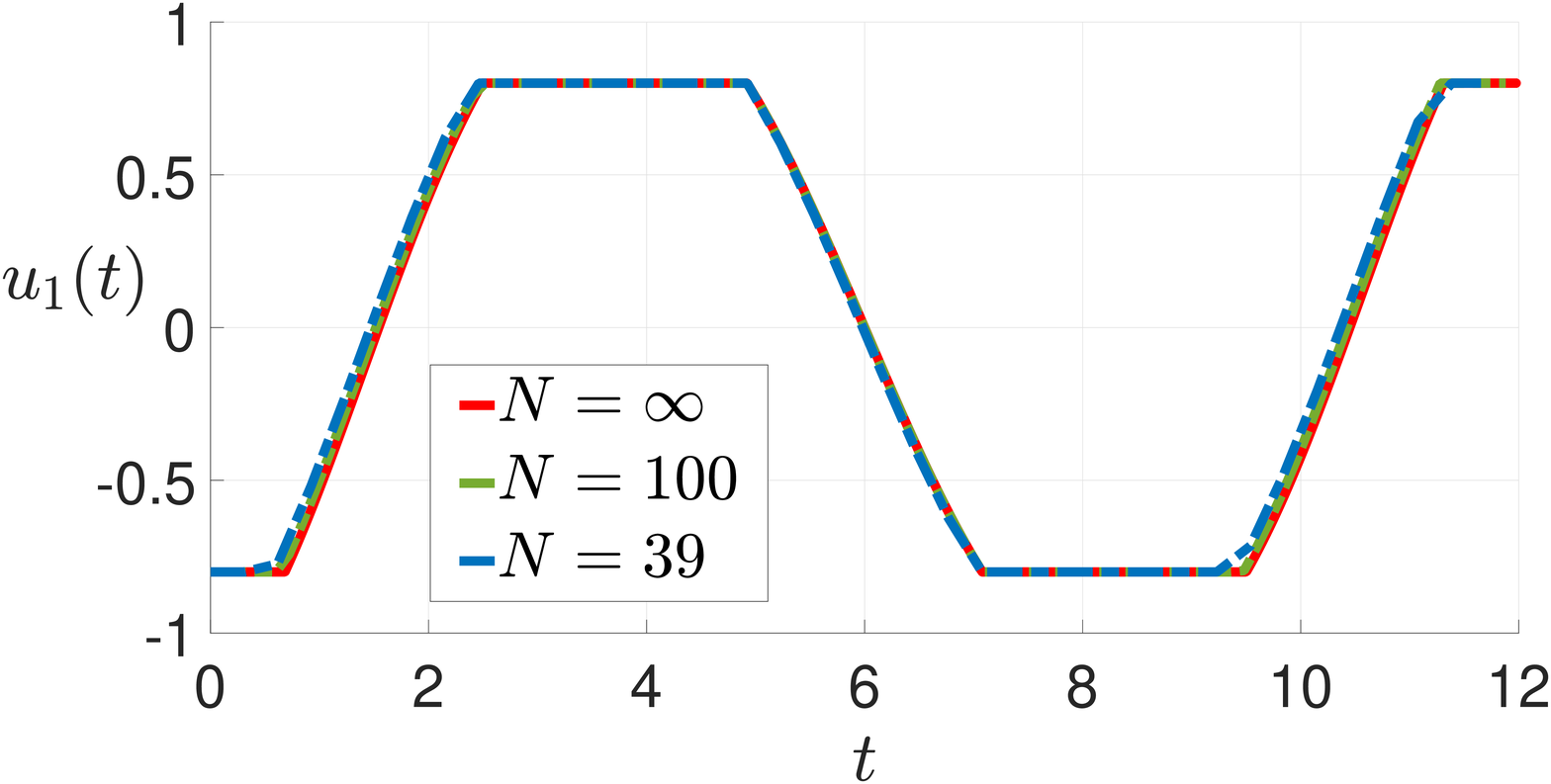}
  \caption{\small $u_1(\cdot)$ by algorithm PDP-2.}
  \label{fig:u1ApproxEg3_PDP}
\end{subfigure}
\begin{subfigure}{.49\linewidth}
\includegraphics[width=\linewidth]{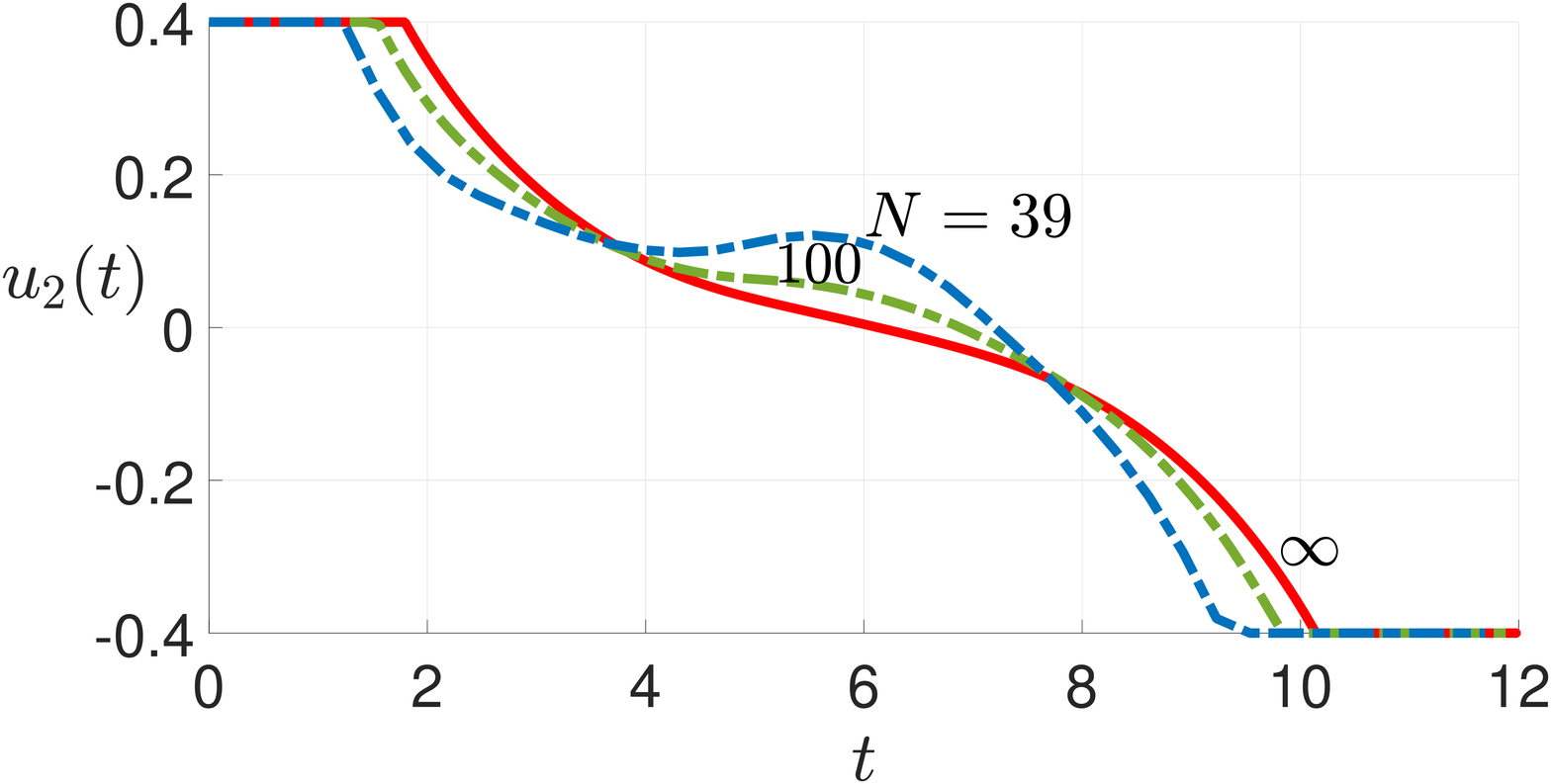}
  \caption{\small $u_2(\cdot)$ by algorithm PDP-2.}
  \label{fig:u2ApproxEg3_PDP}
  \end{subfigure}
  \caption{\sf Problem $(P3)$---Solution $u(\cdot)$ by algorithm PDP-2 under different number of discretization points ($N=39,$ and $100$; the solution when $N\to\infty$ is represented in the red solid-line curve, obtained with $N=1000$).
  \label{fig:uVSnEg3}}
\end{figure}

\section{Conclusion and Discussion}\label{sec:conclusion}

Our work is an application of the primal-dual framework and the deflected subgradient algorithm studied in \cite{BurachikXuemei2023} with a numerical implementation to solve optimal control problems. Hence our framework inherits the theoretical properties of the previous work, such as (i) strong duality (Theorem~\ref{th:StrongDual}), (ii) monotone improvement of the dual function (Proposition~\ref{prop:dualProperty}), and (iii)  every accumulation point of the primal sequence is a solution (for both step-sizes in PDP-1 and PDP-2). Moreover, PDP-2 converges in a finite number of iterations when the dual sequence is bounded (Theorem~\ref{primalCvg}(a)).

We consider infinite dimensional optimization problems which satisfy the assumptions~{\rm\hyperlink{H0}{(H0)}--\hyperlink{H2}{(H2)}}. We have presented here a systematic technique to verify these assumptions in the framework of very general types of optimization problems (Theorem~\ref{th:example_verify_H}). We show this for problems $(P2)$ and $(P3)$ (see Theorem~\ref{th:verify_p1} and Corollary~\ref{lem:DP for (P3)}). Particularly, we have demonstrated how to reformulate the ODE constraints in Problems $(P2)$ and $(P3)$ as equality constraints so that the assumptions~{\rm\hyperlink{H0}{(H0)}--\hyperlink{H2}{(H2)}} can be verified.

Problem $(P2)$ is the constrained optimal control of the double integrator and Problem $(P3)$ is the more challenging free-flying robot, again constrained. We illustrate the iterations of the control variables using our PDP algorithm, using Ipopt in solving its subproblems, for Problem $(P2)$ in Figure~\ref{fig:uCvgEg2} and for Problem $(P3)$ in Figure~\ref{fig:uApproxEg3}. Compared with using Ipopt alone, our PDP algorithm shows a better performance in solving the challenging flying robot problem in terms of the CPU time and in the case of increased number of discretization points (see Table~\ref{table:FFRobot}).

In PDP-1, the step-size $s_k$ as in~\eqref{sk_DSG1} gets smaller in each iteration and the increments in the penalty parameter $c_k$ gets smaller accordingly. In numerical practice, PDP-1 has a much bigger increment in $c_k$ in the first one or two iterates. Hence, PDP-1 can often find the optimal dual solution in just a few iterates. Although PDP-1 approaches  the dual solution very quickly in the first few iterates, it becomes sluggish in the following iterates since $c_k$ is incremented very slowly.

Algorithm PDP-2 uses the step-size $s_k$ as in \eqref{sk_DSG2}. The increase in the penalty parameter $c_k$ is small in the initial iterates, compared to PDP-1. PDP-2 resembles the penalty method with constant increments in $c_k$. In practice, PDP-2 approaches the dual optimal solution slowly but steadily. To avoid the slow progress of PDP-1 in later iterations, we would suggest a  hybrid strategy, which applies PDP-1 in the first few iterates and then switches to PDP-2.

Apart from the examples in this paper, our algorithm should be applicable to more general optimal control problems, for example the more challenging problems involving mixed state-control constraints or even pure state constraints.

The PDP algorithm could be used as a theoretical tool that provides dual information on a given problem. When PDP is applied to solve problems where the analytical solution can be found relatively easily (e.g., the unconstrained double integrator problem), then the use of the algorithm is likely to reveal new properties of the problem, especially those arising from duality.  This is work in progress.

Practically speaking, because we are doing the iterations with discretized functions, there is already a level of inexactness introduced into the subproblem solutions. The numerical experiments show that this kind of inexactness is dealt with successfully. Other types of  inexact versions could be done by extending the work presented in \cite{BKMinexact2010}, which deals with finite dimensional problems. The cases we consider in the present paper are infinite dimensional and hence more challenging. Theoretical investigation of inexactness in infinite dimensions remains an open problem and  hence the topic of future research.

\setcounter{section}{0}
\appendix
\renewcommand{\thesection}{Appendix \Alph{section}}
\section{} \label{sec:proofs}
\setcounter{equation}{0}
\renewcommand{\theequation}{A.\arabic{equation}}

\textbf{\hypertarget{prf:lem_varfi}{Proof of Lemma \ref{lem:varfi}}}.

\begin{proof} 
Assume that $\{u_k\}\subset ({\mathcal L}^2([0,T];\R))^m$ is such that $u_k\rightharpoonup u$. Then for all $j=1,\ldots,m$ we have $u^j_k\rightharpoonup u^j$ (weakly in ${\mathcal L}^2([0,T];\R)$). For $t\in [0,T]$, define $\xi_t(s)=0$ for $s\in (t,T]$ and $\xi_t(s)=1$ for $s\in [0,t]$. Then $\xi_t\in {\mathcal L}^2([0,T];\R)$ and it is easy to check that $\|\xi_t\|_2=\sqrt{t}$. The weak convergence yields
\begin{equation}
    \label{eq:E2}
 \lim_{k\to \infty} \int_{0}^t u^j_k(s) ds= \lim_{k\to \infty} \int_{0}^T u^j_k(s) \xi_t(s) ds=  \int_{0}^T u^j(s) \xi_t(s) ds  = \int_{0}^t u^j(s) ds,
\end{equation}
where we used the definition of $\xi_t$ in the first and last equality, and the assumption of weak convergence in the second one. For each $j=1,\ldots,m$, define $f^j_k,f^j:[0,T]\to \R $ as $f^j_k(t):= \int_{0}^t u^j_k(s) ds$ and $f^j(t):= \int_{0}^t u^j(s) ds$. By \eqref{eq:E2}, 
    \begin{equation}
        \label{eq:E3}
        \lim_{k\to \infty}f^j_k(t)=f^j(t),
    \end{equation}
for every $t\in [0,T]$ and each $j=1,\ldots,m$. We will apply Theorem \ref{th:LDCT} to the sequence $\{f^j_k\}$ for each $j=1,\ldots,m$. Since $\{u_k\}$ converges weakly, the set $U_0:=\{u_k\::\: k\in \mathbb{N}\}\cup \{u\}$ is weakly compact, and hence bounded by Corollary \ref{cor:CCB}. Therefore, there exists $L_0>0$ such that $\|u_k^j\|_2\le L_0$ for all $k\in \mathbb{N}, \, j=1,\ldots,m$. We show next that the sequence $\{f^j_k\}$ is bounded over $[0,T]$ for each $j=1,\ldots,m$. Indeed, for every $t\in [0,T]$ we can use Cauchy-Schwartz to write
\begin{equation}
    \label{eq:fk bded}
    |f^j_k(t)|= \Big|\int_{0}^t u^j_k(s) ds \Big|= |\langle u^j_k,\xi_t\rangle_{{\mathcal L}^2([0,T];\R)}|\le \|u^j_k\|_2 \|\xi_t\|_2 \le L_0 \sqrt{t}\le L_0 \sqrt{T},
\end{equation}
where $\langle \cdot,\cdot \rangle_{{\mathcal L}^2([0,T];\R)}$ denotes the scalar product in ${\mathcal L}^2([0,T];\R)$. This establishes the desired boundedness. Using now Theorem \ref{th:LDCT} we deduce that
\begin{equation}
    \label{eq:int fk}
 \int_{0}^t f^j(s) ds = \lim_{k\to \infty}  \int_{0}^t f_k^j(s) ds,
\end{equation}
which by definition of $f^j_k, f^j$ re-writes as
\begin{equation}
    \label{eq:int fk 2}
 \int_{0}^t \int_{0}^s u^j(r) dr ds = \lim_{k\to \infty}  \int_{0}^t \int_{0}^s u_k^j(r) dr ds,
\end{equation}
Using the definition of $\pi(u)$ and \eqref{eq:int fk 2} we deduce that
\begin{equation*}
    \label{eq:pi}
    \lim_{k\to \infty} \pi(u_k)(t)= \pi(u_k)(t),
\end{equation*}
for every $t\in [0,T]$. Since $\varphi$ is continuous, we further have
\begin{equation}
    \label{eq:varfi}
    \lim_{k\to \infty} \varphi(\pi(u_k)(t))= \varphi(\pi(u_k)(t)),
\end{equation}
for every $t\in [0,T]$. Consider the functions $\omega_k:= \varphi(\pi(u_k))$ and $\omega:= \varphi(\pi(u))$. By definition, $\omega_k,\omega:[0,T]\to \R$. We claim that $\omega_k,\omega\in {\mathcal L}^2([0,T];\R)$. Indeed, the boundedness assumption on $\varphi$ gives
\[
\|\omega_k\|^2_2= \int_{0}^T  |\varphi(\pi(u_k)(s)|^2 ds \le L_1^2\, T.
\]
An identical argument shows that $\omega\in {\mathcal L}^2([0,T];\R)$. Note that this claim implies that  $\omega_k,\omega\in {\mathcal L}^2([0,t];\R)$ for every $t\in [0,T]$. Now we claim that $\omega_k\to \omega$ strongly in ${\mathcal L}^2([0,t];\R)$ for every $t\in [0,T]$.  Namely, we claim that 
\begin{equation}
    \label{claim 2}
     \lim_{k\to\infty}\|\omega_k-\omega\|^2_{{\mathcal L}^2([0,t];\R)}=  \lim_{k\to\infty}\int_0^t | \omega_k(s)-\omega(s)|^2 ds=0,
\end{equation}
 for every $t\in [0,T]$. Indeed, for $s\in [0,T]$ define $\Delta_k(s):= | \omega_k(s)-\omega(s)|^2$. The definition of the functions $\omega_k,\omega$ and \eqref{eq:varfi} imply that $\lim_{k\to\infty} \Delta_k(s) =0$ for every $s\in [0,T]$. Using the boundedness of $\varphi$ and the definitions, we also have that
\[
\Delta_k(s)= |\omega_k(s)|^2 +|\omega(s)|^2+ 2 |\omega_k(s)|\,|\omega(s)|\le 4 L_1^2,
\]
for every $s\in [0,T]$. Now we can apply Theorem \ref{th:LDCT}(b), \eqref{eq:varfi}, and the definitions to deduce that
\[
\lim_{k\to\infty} \int_0^t \Delta_k(s) ds = \lim_{k\to\infty} \int_0^t | \omega_k(s)-\omega(s)|^2 ds =\int_0^t  \lim_{k\to\infty} \Delta_k(s) ds = 0,
\]
for every $t\in [0,T]$. This establishes \eqref{claim 2}. Using the definition of $\xi_t$, the above expression re-arranges as follows.
\[
\begin{array}{rcl}
    0 &= & \lim_{k\to\infty} \int_0^t | \omega_k(s)-\omega(s)|^2 ds =\int_0^T   \xi_t(s) | \omega_k(s)-\omega(s)|^2 ds ,\\
     & &\\
     &&= \int_0^T   (\xi_t(s))^2 | \omega_k(s)-\omega(s)|^2 ds = \int_0^T   |\xi_t(s) \omega_k(s)-\xi_t(s)\omega(s)|^2 ds,
\end{array}
\]
showing that the sequence $\{\xi_t\omega_k\}$ converges strongly in ${\mathcal L}^2([0,T];\R)$ to $\xi_t\omega$.  To complete the proof, we will use Proposition \ref{prop:weak-strong}, for the space $X:={\mathcal L}^2([0,T];\R)$, the strongly convergent sequence $\{\xi_t\omega_k\}$, and each of the weakly convergent sequences $\{u^j_k\}$ for $j=1,\ldots,m$. This proposition implies that
\begin{equation}
    \label{eq:conv-final}
    \begin{array}{rcl}
  \lim_{k\to\infty} \int_0^t  \omega_k(s) u^j_k(s) ds       &  =  &  \lim_{k\to\infty} \int_0^T \xi_t(s) \omega_k(s) u^j_k(s) ds\\
         & &\\
         &=& \int_0^T \xi_t(s) \omega(s) u^j(s) ds =\int_0^t  \omega(s) u^j(s) ds. \\ 
    \end{array}
\end{equation}
the above expression and the definition of $\omega_k,\omega$ imply that \eqref{claim:eta} holds. Finally \eqref{claim:rho} will follow from applying Theorem \ref{th:LDCT}(b) to the sequence $\{ \eta^{\varphi}_j(u_k,\cdot)\}$. Indeed, \eqref{eq:conv-final} means that $\lim_{k\to \infty} \eta^{\varphi}_j(u_k,s)=\eta^{\varphi}_j(u,s) $ for every $s\in [0,T]$. We now use an argument similar to the one in \eqref{eq:fk bded} to show the boundedness of the sequence over $[0,T]$. Indeed, fix $t\in [0,T]$.
\[
\begin{array}{rcl}
 |\eta^{\varphi}_j(u_k,t)|&=&|\int^t  \omega_k(s) u^j_k(s) ds | =
 |\int^T  \xi_t(s)\omega_k(s) u^j_k(s) ds |
 \\
 &&\\
     & =& |\langle u^j_k,\xi_t \omega_k\rangle_{{\mathcal L}^2([0,T];\R)}|\le \|u^j_k\|_2 \|\xi_t\omega_k\|_2 \le L_0 \,L_1\sqrt{t}\le L_0 L_1 \sqrt{T},
\end{array}
\]
Now Theorem \ref{th:LDCT}(b) yields 
\[
 \lim_{k\to\infty} \int_0^t \eta^{\varphi}_j(u_k,s) ds=   \int_0^t \eta^{\varphi}_j(u,s) ds,
\]
which is \eqref{claim:rho}. The proof is complete.
\end{proof}

\newpage
\textbf{\hypertarget{prf:h for P3}{Proof of Lemma~\ref{lem:h for P3}}}

\begin{proof}
Fix $T:=12$. The desired function $h$ will be obtained by repeatedly applying the Fundamental Theorem of Calculus to each equation of the ODE system in Problem $(P3)$, starting with the last equation. Consider the last ODE equation in $(P3)$, together with its boundary conditions, namely
\begin{equation}\label{sys:6}
 \begin{array}{lcr}
 \left\{ 
 \begin{array}{l}
\dot x_6(t)=0.2\left(u_1(t)-u_2(t)\right)=:g_6(u(t)), \\
\\
      x_6(0)=0,\quad    x_6(T)=0,
 \end{array}
\right.
\end{array}
    \end{equation}
By Fundamental Theorem of Calculus this system {is equivalent to} 
\(
G_6(u(T))=0,
\)
where $G_6(u(t)):=\ds\int_0^t g_6(u(s)) ds$. Hence, for every $u\in {\mathcal L}^2([0,T];\R)\times {\mathcal L}^2([0,T];\R)$ we define
\begin{equation*}\label{eq:G6}
h_6(u):= G_6(u(T))=\int_0^T g_6(u(s)) ds  =
a_1 \int_0^T u_1(s) ds - a_2 \int_0^T u_2(s) ds,
\end{equation*}
where we used the definition of $g_6$ in the third equality with the notation $a_1:=a_2:=0.2$. Thus, $h_6:{\mathcal L}^2([0,T];\R)\times {\mathcal L}^2([0,T];\R)\to \R$. In particular, we can write the solution $x_6(\cdot)$ of system \eqref{sys:6} as a function of $u$. Namely,
\begin{equation}\label{eq:x6}
 x_6(t)=\int_0^t g_6(u(s)) ds  =
a_1 \int_0^t u_1(s) ds - a_2 \int_0^t u_2(s) ds.   
\end{equation}
Since $x_6$ is a function of $u$, a similar procedure can be used for the third equation in $(P3)$ and its boundary conditions. Indeed, using \eqref{eq:x6}

\begin{equation}
 \label{sys:3}
\begin{array}{llll}
\hbox{the system}&
\left\{\begin{array}{l}
  \dot x_3(t)  =x_6(t)  = \int_0^t g_6(u(s)) ds , \\
   x_3(0)=\pi/2,\, x_3(T)=0,
\end{array}
\right\},
     & \hbox{is equivalent to} &
G_3(u(T))=0,
 \\
\end{array}
\end{equation}
where $G_3(u(t)):=\pi/2+ \ds\int_0^t \int_0^r g_6(u(s)) ds dr$ and $g_6$ is defined in the previous system. We can thus write, for every $u\in {\mathcal L}^2([0,T];\R)\times {\mathcal L}^2([0,T];\R)$
\begin{equation*}\label{eq:G3a}
    \begin{array}{rcl}
h_3(u):=G_3(u(T))     &=  & \pi/2+\ds\int_0^T \int_0^r g_6(u(s)) ds \,dr \\
&&\\
     &  &= \pi/2+ a_1 \ds\int_0^T\int_0^r u_1(s) ds \, dr - a_2  \int_0^T \int_0^r u_2(s) ds\, dr.
\end{array}
\end{equation*}
Hence, $h_3:{\mathcal L}^2([0,T];\R)\times {\mathcal L}^2([0,T];\R)\to \R$.  We proceed now to define the remaining $h_i$'s. As a consequence of our last construction, we have that we can write the solution $x_3(\cdot)$ of system \eqref{sys:3} as a function of $u$. Namely,
 $x_3(t)=G_3(u(t))$. Therefore, we have the following equivalence
\[
\begin{array}{lll}

\left\{\begin{array}{l}
  \dot x_5(t)    = \left(u_1(t)+u_2(t)\right)\sin x_3(t)\\
  \\
 \;\; = \left(u_1(t)+u_2(t)\right)\sin G_3(u(t))=:g_5(u(t)), \\
  \\
   x_5(0)=0,\, x_5(T)=0,
\end{array}
\right\},
     & \hbox{if and only if} &
G_5(u(T))=0,
 \\
\end{array}
 \]
where $G_5(u(t)):=\ds\int_0^t g_5(u(s)) ds$. As above, use the definition of $g_5$ to define
\begin{equation*}
   \label{def:h5} 
    \begin{array}{rcl}
\hspace{-2mm}h_5(u):= G_5(u(T))    &=   &\ds\int_0^T g_5(u(s)) ds = \ds\int_0^T u_1(s) \sin G_3(u(s))   ds +  \ds\int_0^T u_2(s) \sin G_3(u(s)) ds.
    \end{array}
\end{equation*}
Using again the fact that $x_3(t)=G_3(u(t))$, write
\[
\begin{array}{lll}

\left\{\begin{array}{l}
  \dot x_4(t)    = \left(u_1(t)+u_2(t)\right)\cos x_3(t)\\
  \\
 \;\; = \left(u_1(t)+u_2(t)\right)\cos G_3(u(t))=:g_4(u(t)), \\
  \\
   x_4(0)=0,\, x_4(T)=0,
\end{array}
\right\},
     & \hbox{is equivalent to} &
G_4(u(T))=0,
 \\
\end{array}
 \]
where $G_4(u(t)):=\ds\int_0^t g_4(u(s)) ds$. So we define
\begin{equation*}
   \label{def:h4} 
h_4(u):= G_4(u(T)):=\int_0^T g_4(u(s)) ds =
\int_0^T\left(u_1(t)+u_2(t)\right)\cos G_3(u(t)) dt. 
\end{equation*}
Again, note that the system
\[
\begin{array}{lll}
\left\{\begin{array}{l}
  \dot x_2(t)    = x_5(t) = \ds\int_0^t g_5(u(s))ds =:g_2(u(t)), \\
  \\
   x_2(0)=-10,\, x_2(T)=0,
\end{array}
\right\},
     & \hbox{is equivalent to} &
G_2(u(T))=0,
 \\
\end{array}
 \]
where $G_2(u(t)):=-10+\ds\int_0^t \int_0^r g_5(u(s))ds dr$. So we define 
\[
h_2(u):= G_2(u(T)):=-10+\ds\int_0^T g_2(u(s)) ds.
\] 
Finally, we can write
\[
\begin{array}{lll}

\left\{\begin{array}{l}
  \dot x_1(t)    = x_4(t) = \ds\int_0^t g_4(u(s))ds =:g_1(u(t)), \\
  \\
   x_1(0)=-10,\, x_1(T)=0,
\end{array}
\right\},
     & \hbox{is equivalent to} &
G_1(u(T))=0,
 \\
\end{array}
 \]
where $G_1(u(t)):=-10+\ds\int_0^t \int_0^r g_4(u(s))ds dr$.  So we define \\ $h_1(u):=-10+\ds\int_0^T g_1(u(s)) ds$. Altogether, the ODE system in $(P3)$ can be rewritten in terms of $u$ as
\begin{equation}\label{app:h}
  h(u)=(G_1(u(T)),\ldots,G_6(u(T)))=(h_1(u),\ldots,h_6(u))=0\in \R^6.
\end{equation}
\end{proof}

\textbf{\hypertarget{prf:cond b for P3}{Proof of Lemma~\ref{lem:cond b for P3}}}

\begin{proof}
The w-closedness will follow from the w-compactness and Fact \ref{lem:Weakly Closed}(b). We proceed to establish the w-compactness. For $h$ as in Lemma \ref{lem:h for P3}, write
   \[
\Gamma(z):=\{u\in K_2\::\: h(u)=z\}= h^{-1}(z)\cap K_2=\bigcap_{j=1}^6 \left[h_j^{-1}(z_j)\cap K_2\right].
\] 
Call $\Gamma_j:=h_j^{-1}(z_j)\cap K_2$ for $j=1,\ldots,6$. We will show that each $\Gamma_j$ is weakly compact for $j=1,\ldots,6$. Indeed, with the notation of Lemma \ref{lem:varfi}, and the definition of $h_j$ given in Lemma \ref{lem:h for P3}, we have that
\begin{equation}\label{eq:eta-ro}
 \begin{array}{rcll}
   h_1(u)  &=  &  -10 +\rho_1^{\varphi_c}(u,T)+\rho_2^{\varphi_c}(u,T),
    
   &h_2(u)  = -10 + \rho_1^{\varphi_s}(u,T)+\rho_2^{\varphi_s}(u,T),
\\
&&&\\
  h_3(u)  & =& \pi/2+ \rho_1^{\varphi_1}(u,T)-\rho_2^{\varphi_1}(u,T),
       &  h_4(u)   = \eta_1^{\varphi_c}(u,T)+\eta_2^{\varphi_c}(u,T),\\
 &&&\\ 
  h_5(u)  & =&  \eta_1^{\varphi_s}(u,T)+\eta_2^{\varphi_s}(u,T),
& h_6(u)   = 0.2\,( \eta_1^{\varphi_1}(u,T)-\eta_2^{\varphi_1}(u,T)),
\end{array}
\end{equation}
where $\varphi_c:=\cos(\cdot)$, $\varphi_s:=\sin(\cdot)$ and $\varphi_1(r)=1$ for every $r\in [0,T]$. Fix $j\in \{1,\ldots,m\}$. By Theorem \ref{th:ES}, $\Gamma_j$ is weakly compact if and only if it is sequentially weakly compact. Namely, if and only if, for every sequence $\{u_k\}\subset \Gamma_j$, there exists a subsequence $\{u_{k_l}\}\subset \{u_k\}$ s.t. $u_{k_l}\rightharpoonup u\in \Gamma_j$.  Take now any sequence $\{u_k\}\subset \Gamma_j$. Since $\Gamma_j\subset K_2$ and $K_2$ is weakly compact, there exists a subsequence 
$\{u_{k_l}\}\subset \{u_k\}$  s.t. $u_{k_l}\rightharpoonup u\in K_2$. Because  $\{u_{k_l}\}\subset \Gamma_j$ we have that $h_j(u_{k_l})=z_j$. By \eqref{claim:eta} and \eqref{claim:rho} in Lemma \ref{lem:varfi} we deduce from \eqref{eq:eta-ro} that 
\[
z_j= \lim_{l\to \infty} h_j(u_{k_l})= h_j(u),
\]
for $j=1,\ldots,6$. Hence, $u\in \Gamma_j$ as wanted and $\Gamma_j$ is weakly compact for all $j=1,\ldots,m$. Therefore $\Gamma(z)$ is w-compact and hence w-closed.
\end{proof}

\def\cprime{$'$}

\section*{Acknowledgments}
The authors offer their warm thanks to the Editor for their efficient handling of the paper.  They are also indebted to an anonymous reviewer whose comments improved the manuscript.  Xuemei Liu was supported by an Australian Government Research Training Program Scholarship.

\end{document}